\documentclass[11pt]{article}

\usepackage{multicol}

\usepackage{multicol}

\usepackage{amsfonts,amsmath,amssymb,amsthm,enumerate}

\usepackage{color,tikz,calc,graphicx,fullpage}

\usepackage{ifthen,latexsym}

\usepackage[utf8]{inputenc}

\usepackage[T1]{fontenc}

\usepackage[colorlinks=true,citecolor=black,linkcolor=black,urlcolor=blue]{hyperref}

\usepackage{setspace,graphicx}

\usepackage{url,epsfig,ifpdf}

\ifpdf 

\usepackage{twoopt,mathtools,stmaryrd}

\usepackage[numbers]{natbib}

\renewcommand{\Pr}{{\bf{P}}}

\newcommand{\cond}{{ \; \big\vert \; }}

\newcommand{\cA}{\mathcal{A}}

\newcommand{\cD}{\mathcal{D}}

\newcommand{\cE}{\mathcal{E}}

\newcommand{\cF}{\mathcal{F}}

\newcommand{\cG}{\mathcal{G}}

\newcommand{\cP}{\mathcal{P}}

\newcommand{\eps}{\varepsilon}

\renewcommand{\epsilon}{\varepsilon}

\newcommand{\Nat}{\mathbb{N}}

\newcommand{\Real}{\mathbb{R}}

\DeclareMathOperator*{\E}{\scalebox{1.3}{$\mathbb E$}}

\DeclareMathOperator{\supp}{supp}

\DeclareMathOperator{\ex}{ex}

\DeclareMathOperator{\Hom}{Hom}

\newcommand{\unlab}[2]{\left\llbracket #1\right\rrbracket_{#2}}

\newcommand{\wbar}{\widetilde}

\def\webpage{{http://honza.ucw.cz/proj/codeg-k4e/}}

\newcommand{\vc}[1]{\ensuremath{\vcenter{\hbox{#1}}}}

\newcommandtwoopt{\vcc}[3][1][0.7]{\vc{\includegraphics[scale=#2,page=#1]{#3}}}

\newtheorem{theorem}{Theorem}

\newtheorem{definition}[theorem]{Definition}

\newtheorem{lemma}[theorem]{Lemma}

\newtheorem{corollary}[theorem]{Corollary}

\newtheorem{conjecture}[theorem]{Conjecture}

\newtheorem{claim}[theorem]{Claim}

\newtheorem{question}[theorem]{Question}

\newtheorem{problem}[theorem]{Problem}

\newtheorem{proposition}[theorem]{Proposition}

\newtheorem{observation}[theorem]{Observation}

\newtheorem{construction}[theorem]{Construction}

\numberwithin{theorem}{section}

\newcommand{\beq}{\begin{equation}}

\newcommand{\eeq}{\end{equation}}

\numberwithin{equation}{section}

\numberwithin{theorem}{section}

\renewcommand{\qed}{\nolinebreak\mbox{\hspace{5 true pt}%

  \rule[-0.85 true pt]{3.9 true pt}{8.1 true pt}}}

\newcommand{\floor}[1]{\left\lfloor #1 \right\rfloor}

\newcommand{\ceil}[1]{\left\lceil #1 \right\rceil}

\begin{document}

\title{\bf{The codegree threshold of $K_4^{-}$}}

\author{Victor Falgas-Ravry\thanks{Research supported by Swedish Research Council grants VR 2016-03488 and VR 2021-03687.} \\
\small Institutionen f\"or matematik och matematisk statistik\\
\small Ume{\aa} Universitet, Sweden\\
\small \tt victor.falgas-ravry@umu.se\\
\and
Oleg Pikhurko\thanks{Supported by Leverhulme Research Project Grant RPG-2018-424 and ERC Advanced Grant 101020255.}\\
\small Mathematics Institute and DIMAP\\
\small University of Warwick, UK\\
\small \tt O.Pikhurko@Warwick.ac.uk \\
\and
Emil Vaughan\\
\small Citymapper, UK\\
\small \tt emil79@gmail.com
\and Jan Volec\\
\small Department of Mathematics\\
\small Czech Technical University in Prague, Czech Republic\\
\small Faculty of Nuclear Sciences and Physical Engineering\\
\small \tt jan@ucw.cz}

\date{}

\maketitle

\begin{abstract}

The codegree threshold $\ex_2(n, F)$ of a $3$-graph $F$ is the minimum $d=d(n)$ such that every $3$-graph on $n$ vertices in which every pair of vertices is contained in at least $d+1$ edges contains a copy of $F$ as a subgraph. We study $\ex_2(n, F)$ when $F=K_4^-$, the $3$-graph on $4$ vertices with $3$ edges. Using flag algebra techniques, we prove that if $n$ is sufficently large then
\[\ex_2(n, K_4^-)\leq \frac{n+1}{4}.\]
This settles in the affirmative a conjecture of Nagle~\cite{Nagle99}. In addition, we obtain a stability result: for every near-extremal configuration $G$, there is a quasirandom tournament $T$ on the same vertex set such that $G$ is $o(n^3)$-close in the edit distance to the $3$-graph $C(T)$ whose edges are the cyclically oriented triangles from $T$. For infinitely many values of $n$, we are further able to determine $\ex_2(n, K_4^-)$ exactly and to show that tournament-based constructions $C(T)$ are extremal for those values of $n$.\\
\textbf{Mathematics classification codes:} 05D99; 05C65; 05C20. 
\end{abstract}

\section{Introduction}\label{section: introduction}

Interest in the extremal theory of hypergraphs dates back to Tur\'an's celebrated 1941 paper~\cite{Turan41}. However, despite significant efforts from the research community, the problem of determining the Tur\'an density of a given hypergraph $F$ is open in all but a small number of cases --- see e.g.\  Keevash's survey of the field~\cite{Keevash11}. The difficulty of the problem has led researchers to investigate a number of other notions of extremal density, notably the codegree density, which is studied in this paper. Before we describe our results in detail, we introduce some basic notation.

\subsection{Notation}
Given a set $A$, we write $A^{(r)}$ for the collection of all unordered $r$-tuples from $A$. Write $[n]$ for the set $\{1,2,\ldots n\}$. A \emph{$3$-graph} or \emph{triple system} is a pair $G=(V,E)$, where $V=V(G)$ is a set of vertices and $E=E(G)\subseteq V^{(3)}$ is a collection of unordered triples, which constitute the edges of $G$. We set $v(G):=\vert V(G)\vert$ and $\vert G\vert:=\vert E(G)\vert$. For notational convenience, we shall often identify a $3$-graph with its edge-set, and write $x_1x_2\ldots x_r$ to denote the set $\{x_1,x_2, \ldots x_r\}$. Given two $3$-graphs $G$ and $G'$ on a common vertex set $V$, their \emph{edit distance} $\vert G\Delta G'\vert$ is the size of the symmetric difference of their edge-sets. The \emph{link graph} of a vertex $x$ in a $3$-graph $G$ is 
\[G_x:=(V(G)\setminus \{x\}, \{yz: \ xyz\in E(G)\}),\]
and the  \emph{joint neighbourhood} of a pair $\{x,y\}$ is 
\[G_{xy}:=\{z: \ xyz \in E(G)\}.\]

A \emph{subgraph} of $G$ is a $3$-graph $H$ with $V(H)\subseteq V(G)$ and $E(H)\subseteq E(G)$. If $G$ does not contain a copy of $F$ as a subgraph, we say that $G$ is \emph{$F$-free}. The \emph{Tur\'an number} $\ex(n, F)$ of a non-empty $3$-graph $F$ is the maximum number of edges in an $F$-free $3$-graph on $n$ vertices, and its \emph{Tur\'an density} is the limit $\pi(F):=\lim_{n\rightarrow \infty} \ex(n,F)/\binom{n}{3}$ (this is easily shown to exist). In this paper we shall be interested in variants of the Tur\'an number and the Tur\'an density.

The \emph{codegree} of a pair $xy\in V(G)^{(2)}$ is $d(x,y):=\vert G_{xy}\vert$, the number  of edges of $G$ containing the pair $xy$. The \emph{minimum codegree} of $G$, which we denote by $\delta_2(G)$, is the minimum of $d(x,y)$ over all pairs $xy\in V(G)^{(2)}$. The \emph{codegree threshold} $\ex_2(n,F)$ of a non-empty $3$-graph $F$ is the maximum of $\delta_2(G)$ over all $F$-free $3$-graphs on $n$ vertices. A probabilistic averaging argument~\cite{MubayiZhao07} yields that the limit 
\[\pi_2(F):=\lim_{n\rightarrow \infty} \frac{\ex_2(n,F)}{n-2}\] 
exists; this quantity is called the \emph{codegree density} of $F$. Another straightforward averaging argument shows that $0\leq \pi_2(F)\leq \pi(F)\leq 1$, and it is known that $\pi_2(F)\neq \pi(F)$ in general (see also Section~\ref{subsection: background}).

In addition to $3$-graphs, we shall also need to consider tournaments in this paper. An \emph{oriented graph} is an ordinary graph together with an orientation of its edges. A \emph{tournament} is an orientation of a complete graph. In an oriented graph $O$, we denote by $N_O^-(x) $ and $N_O^+(x)$ the \emph{in-neighbourhood} and \emph{out-neighbourhood} of a vertex $x$ respectively, that is, the collection of $y$ such that the edge $\{x,y\}$ is oriented into $x$ (as $\vec{yx}$) and out of $x$ (as $\vec{xy}$) respectively. We write $d_O^-(x):=\vert N_O^-(x) \vert$ and $d_O^+(x):= \vert N_O^+(x)\vert$ for the \emph{in-degree} and \emph{out-degree} of $x$. Further, we write $d_O^-(x, Y)$ and $d_O^+(x, Y)$ for the number of in- and out-edges of $x$, respectively, with the other endpoint being inside the set $Y$. Note that when the oriented graph $O$ is clear from the context, we omit the subscript in the in- and out-degree notation.
\begin{figure}
\hfill\foreach \n in {21,22}{ \includegraphics[page=\n]{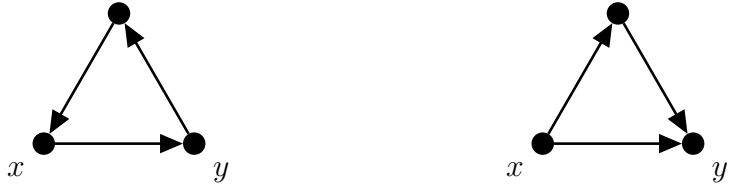}\hfill}

\caption{Configurations in a tournament $T$ on the left and right counted by $C_T(x,y)$ and $R_T(x,y)$, respectively.}

\label{fig:C(xy) R(xy)}

\end{figure}
For a tournament $T$ and an arc $\vec{xy} \in E(T)$, we write $C_T(x,y)$ to denote the number of cyclically oriented triangles in $T$ that contains both $x$ and $y$, i.e., the number of vertices $z \in V(T)$ such that both $\vec{yz}$ and $\vec{zx}$ are arcs in $T$. Similarly, we let 
\[R_T(x,y) := \left\vert \{z \in V(T): \{\vec{xz},\vec{zy}\} \in E(T)  \} \right\vert.\] 
Note that $R_T(x,y)$ is equal to the number of cyclically oriented triangles containing both of $x$ and $y$ in the tournament $T'$ obtained from $T$ by reversing the orientation of $\vec{xy}$.

Motivated by the relation between tournaments and $K_4^-$-free $3$-graphs described in Construction~\ref{construction: C(T)} below, we define a notion of \emph{codegree} for tournaments: given a tournament $T$, we define its minimum codegree to be
\[\delta_2(T) := \min\limits_{xy\in V(T)^{(2)}} C_T(x,y).\]
We shall use standard Landau notation throughout this paper: for functions $f,g:\ \mathbb{N}\rightarrow \mathbb{R}_{\geq 0}$ we write $f=o(g)$ if $f(n)/g(n)\rightarrow 0$ as $n\rightarrow \infty$, $f=O(g)$ if $\limsup_{n\rightarrow\infty} \vert f(n)/g(n)\vert $ is finite, and $f=\Omega(g)$ if $g=O(f)$. Finally, we write $f=\Theta(g)$ if $f=\Omega(g)$ and $f=O(g)$ both hold.
\subsection{Results}
In this paper we study the codegree threshold of $K_4^-:=([4], \{123, 124, 134\})$. This is the unique (up to isomorphism) $3$-graph on $4$ vertices with $3$ edges, or, alternatively, this is the complete $3$-graph on $4$ vertices with one edge removed. From the perspective of Tur\'an-type problems, the $3$-graph $K_4^-$ is the smallest non-trivial $3$-graph. Moreover  $K_4^-$-free $3$-graphs have a nice interpretation in terms of their link graphs:
a $3$-graph is $K_4^-$-free if and only if its link graphs are triangle-free. Determining the codegree threshold for $K_4^-$ can thus be viewed as a $3$-graph generalisation of the minimum degree version of Mantel's theorem. In 1999, Nagle~\cite{Nagle99} made the following conjecture (see also~\cite{CzygrinowNagle01}):
\begin{conjecture}[Nagle]\label{conjecture: Nagle99}
$\pi_2(K_4^-)=1/4$.
\end{conjecture}
\noindent The lower bound in Nagle's conjecture comes from an old construction due to Erd{\H o}s and Hajnal~\cite{ErdosHajnal72}:
\begin{construction}[Erd\H{o}s--Hajnal tournament construction]\label{construction: C(T)}
Given a tournament $T$ on the vertex set $V$, define a $3$-graph $C(T)$ on the same vertex set by setting $E(C(T))$ to consist of the elements of $V^{(3)}$ that induce a cyclically oriented triangle in $T$. 
\end{construction}
\noindent It can be easily checked that no $4$-vertex tournament contains more than two cyclically oriented triangles, whence $C(T)$ is a $K_4^-$-free $3$-graph. Also note that $\delta_2(T) = \delta_2(C(T))$. Moreover, if the tournament $T$ is chosen uniformly at random then $\delta_2(T)=n/4 +o(n)$ with high probability.

In this paper we settle Nagle's conjecture in the affirmative:
\begin{theorem}[Codegree density]\label{theorem: codegree density}
$\pi_2(K_4^-)=1/4$.
\end{theorem}
\noindent Our proof of Theorem~\ref{theorem: codegree density} relies on flag algebra techniques: using the semi-definite method of Razborov~\cite{Razborov10}, 
we establish an asymptotic identity (Lemma~\ref{lemma: flag}) between densities of subgraphs on at most seven vertices in $K_4^-$-free $3$-graphs, from which we deduce Nagle's conjecture. Further, by analysing this identity, we are able to show that all near-extremal configurations look like the random tournament construction described above. To make this more precise, let us recall one of the many equivalent definitions of a \emph{quasirandom} tournament (for other forms, see Chung and Graham~\cite{ChungGraham91}).
\begin{definition}
	A tournament $T$ on $[n]$ is \emph{$\delta$-quasirandom} if for every pair of sets $X, Y\subseteq [n]$ we have
	\[\sum_{x\in X} \vert d^+(x, Y) -d^-(x,Y)\vert \leq \delta n^2.\]
\end{definition}
\begin{theorem}[Stability]\label{theorem: stability}
Let $G$ be a $K_4^-$-free $3$-graph on $[n]$ with $\delta_2(G)\geq n/4-o(n)$. Then there exists a $o(1)$-quasirandom tournament $T$ on $[n]$ such that $\vert G\Delta C(T)\vert =o(n^3)$.
\end{theorem}
Finally, we are able to show that tournament-based constructions are extremal for large $n$ and to calculate the exact value of the codegree threshold for infinitely many $n$.
A \emph{skew Hadamard matrix} of order $n$ is an $n\times n$ square matrix $A$ with $\pm1$ entries such that (i) $AA^t= nI_n$, and (ii) $A+A^t=2I_n$. Here $I_n$ denotes the $n\times n$ identity matrix and $A^t$ denotes the transpose of the matrix $A$.
\begin{theorem}[Codegree threshold]\label{theorem: codegree threshold}
For all $n$ sufficiently large,
\[\ex_2(n, K_4^-)\leq \floor{\frac{n+1}{4}}.\]
Further, for all $k$ sufficiently large if there exists a skew Hadamard matrix of order $4k+4$, then for $n=4k+3$ and $n=4k+2$ we have equality in the equation above. Moreover, for all $k$ sufficiently large if $n=4k+3$ and $\ex_2(n, K_4^-)= \floor{\frac{n+1}{4}}$, then every extremal $3$-graph for that value of $n$ is given by an Erd{\H o}s--Hajnal tournament construction and there  exists a skew Hadamard matrix of order $n+1=4k+4$.
\end{theorem}
In~\cite{Wallis}, Seberry (n\'ee Wallis) conjectured a strengthening of Hadamard's conjecture stating that there exists a skew Hadamard matrix for any order $n$ divisible by $4$. It is well-known that every (not necessarily skew) Hadamard matrix must have order $1$, $2$ or a multiple of $4$. Seberry's conjecture is known to hold for $n<276$, see~\cite{Dokovic}, and for all values of $n$ that are of the form $n= 2^t \prod_{i\in I}(q_i+1)$, where $t\in \mathbb{Z}_{\geq 0}$, $I\neq\emptyset$ and $q_i$ is a prime power congruent to $3$  modulo $4$ for every $i\in I$, see~\cite[Theorem~4.1]{WallisStreetSeberry72}.

\begin{corollary}\label{corollary: seberry implies codegree threshold}
If Seberry's conjecture is true, then for all $n$ sufficiently large
\[\ex_2(n, K_4^-)=\left\{\begin{array}{ll} \floor{\frac{n+1}{4}} & \textrm{if }n\equiv 2,3\pmod 4,\\
\floor{\frac{n+1}{4}} \textrm{ or } \floor{\frac{n-3}{4}} & \textrm{if }n\equiv 0,1 \pmod 4.
\end{array}
 \right.\] 
\end{corollary}

\subsection{Background}\label{subsection: background}
In the late 1990s, Nagle~\cite{Nagle99} and Czygrinow and Nagle~\cite{CzygrinowNagle01} made conjectures on the values of the codegree densities $\pi_2(K_4^-)$ and $\pi_2(K_4)$ respectively, where $K_4$ denotes the complete $3$-graph with $4$ vertices.
Mubayi~\cite{Mubayi05} determined the codegree density of the Fano plane, and Keevash and Zhao~\cite{KeevashZhao07} later extended Mubayi's work to other projective geometries. The precise codegree threshold of the Fano plane was determined for large enough $n$ by Keevash~\cite{Keevash09} using hypergraph regularity, and DeBiasio and Jiang~\cite{DeBiasioJiang12} later gave a second, regularity-free proof of the same result. Mubayi and Zhao~\cite{MubayiZhao07} established a number of theoretical properties of codegree density, while Falgas--Ravry~\cite{FalgasRavry13} gave evidence that codegree density problems for complete $3$-graphs may not be stable in general. Falgas-Ravry, Marchant, Pikhurko and Vaughan~\cite{FalgasRavryMarchantPikhurkoVaughan15} for their part determined the codegree threshold of the $3$-graph $F_{3,2}=\{abc, abd, abe, cde\}$ for all $n$ sufficiently large. Finally, Lo and Zhao~\cite{LoZhao18} determined the asymptotic order of $1-\pi_2(K_t^{(3)})$ as $t\rightarrow \infty$, where $K_t^{(3)}$ denotes the complete $3$-graph on $t$ vertices.

In this paper, we add a new example to this scant list of known non-trivial codegree densities by showing $\pi_2(K_4^-)=1/4$. As the smallest non-trivial $3$-graph from the perspective of Tur\'an--type problems, $K_4^-$ has received extensive attention from researchers in the area. Its Tur\'an density is not known, but is conjectured by Mubayi~\cite{Mubayi03} to be $2/7=0.2857\ldots$, with the lower bound coming from a recursive construction of Frankl and F\"uredi~\cite{FranklFuredi84}. Matthias~\cite{Matthias94} and Mubayi~\cite{Mubayi03} proved upper bounds on $\pi(K_4^-)$, before the advent of Razborov's flag algebra framework~\cite{Razborov07}, and in particular his semi-definite method, led to computer-aided improvements by Razborov~\cite{Razborov10} and Baber and Talbot~\cite{BaberTalbot12},  with the current best upper bound for $\pi(K_4^-)$ being $0.2868\ldots$, see~\cite{FalgasRavryVaughan13}.

In addition, `smooth' variants of the Tur\'an density problem for $K_4^-$ have been studied.
Given $\delta > 0$, the \emph{$\delta$-linear density} of a $3$-graph $G$ is the minimum edge-density attained by an induced subgraph of $G$ on at least $\delta v(G)$ vertices. Erd{\H o}s and S\'os~\cite{ErdosSos82} asked whether there is $\delta > 0$ such that every large enough $3$-graph with positive $\delta$-linear density contains a copy of $K_4^-$.
F\"uredi observed however that the tournament construction $C(T)$ of Erd{\H o}s and Hajnal~\cite{ErdosHajnal72} described in the previous section with $T$ chosen at random gives a negative answer to this question: a density of more than $1/4$ is required for the existence of a $K_4^-$-subgraph. Glebov, Kr\'al' and Volec~\cite{GlebovKralVolec16} showed this $1/4$ lower bound is tight, using flag algebraic techniques amongst other ingredients in their proof. More recently, Reiher, R\"odl and Schacht~\cite{ReiherRodlSchacht16} reproved Glebov, Kr\'al' and Volec's result and established the edge-density at which weakly quasirandom $3$-graphs must contain a copy of $K_4^-$, for various notions of `weakly quasirandom'. The extremal problem for $K_4^-$ under both a codegree and a smoothness assumption had been studied earlier by Kohayakawa, R\"odl and Szemer\'edi (see~\cite{Nagle99,ReiherRodlSchacht16}).

\subsection{Organization of the paper}

In Section~\ref{section: flag algebras} we introduce our setting of the flag algebra framework, establish the key flag algebraic identity (Lemma~\ref{lemma: flag}) and prove Theorem~\ref{theorem: codegree density}. In Section~\ref{section: stability}, we extract some information about near-extremal configurations from the flag algebraic identity in order to show their structure must be close to that of a random tournament construction (Theorem~\ref{theorem: stability}).

After studying the relation between tournaments and skew Hadamard matrices in Section~\ref{section: Tournament vs Hadamard}, we devote Section~\ref{section: codegree threshold} to determining the codegree threshold of $K_4^-$ and prove Theorem~\ref{theorem: codegree threshold}. We conclude the paper with some remarks and open problems presented in Section~\ref{section: conclusion}.

\section{Flag algebras and the codegree density of $K_4^-$}\label{section: flag algebras}

Our proof of Theorem~\ref{theorem: codegree density} uses the flag algebra framework introduced by Razborov in~\cite{Razborov07}, and in particular the semi-definite method first deployed by Razborov in~\cite{Razborov10}; see also~\cite{BaberTalbot12, FalgasRavryMarchantPikhurkoVaughan15, FalgasRavryVaughan13, GlebovHefetzLinialMorgenstern22} for expositions of the basic ideas. Such an approach is by now well established in extremal hypergraph theory, and since a treatment of the general theory of flag algebras is outside the scope of this article, we content ourselves here with giving brisk definitions of some of the standard terms and concepts of flag algebras that we shall use, and refer an interested reader to the papers cited above for further details and discussion.

Let $\cF$ denote the set of all non-isomorphic finite $K_4^-$-free $3$-graphs, and let $\cF_k$ denote the subset of $\cF$ consisting of all non-isomorphic $k$-vertex $K_4^-$-free $3$-graphs. A \emph{type} $\sigma$ is an element of $\cF$ together with a labelling of its vertices, i.e.\   a bijection from $[v(\sigma)]$ to $V(\sigma)$. For a fixed type $\sigma$, we define the set $\cF^\sigma$ to be the collection of all (up to $\sigma$-preserving isomorphism) finite $K_4^-$-free $3$-graphs with a fixed embedding of $\sigma$. The elements of $\cF^\sigma$ are referred to as \emph{$\sigma$-flags}. A $\sigma$-flag can be thought of as a $3$-graph on a partially labelled vertex set, with the labelled vertices inducing a copy of $\sigma$. Note that any $3$-graph can be viewed as an $\emptyset$-flag, where $\emptyset$ is the empty type on $0$ vertices. Analogously to the unlabelled case, we let $\cF^\sigma_k$ denote the set of all $k$-vertex $\sigma$-flags. For a $\sigma$-flag $F$, we define its \emph{root} to be the fixed embedding of $\sigma$ in $F$.

Fix now an $\ell$-vertex type $\sigma$. Given two $\sigma$-flags $F$ and $G$, we let $p(F, G)$ to be the probability that a random extension of the root of $G$ by $v(F)-\ell$ unlabelled vertices of $G$ yields a $\sigma$-flag isomorphic to $F$. In the degenerate case $v(F) > v(G)$, we define $p(F,G)=0$. We refer to $p(F,G)$ as the \emph{$\sigma$-flag density} of $F$ in $G$. Further, given three $\sigma$-flags $F_1, F_2$ and $G$ such that $v(G) = v(F_1) + v(F_2) - \ell$, we let $p(F_1,F_2,G)$ denote the probability that a randomly chosen set of $v(F_1)-\ell$ unlabelled vertices of $G$ extends the root of $G$ to a $\sigma$-flag isomorphic to $F_1$, while the remaining $v(F_2)-\ell$ unlabelled vertices extend the root of $G$ to a $\sigma$-flag isomorphic to $F_2$.

We are now ready to describe the \emph{flag algebra} $\cA^\sigma$. Informally, $\cA^{\sigma}$ is obtained by taking $\Real \cF^\sigma$, i.e., the vector space of all formal finite linear combinations of elements of $\cF^{\sigma}$, and for every $F\in \cF^{\sigma}$ and $k\geq v(F)$ quotienting out the relation
\[\sum\limits_{G \in \cF_k^{\sigma}} p(F, G) \cdot G = F.\]
We then define a multiplication in $\cF^\sigma$ by setting
\[F_1 \times F_2 := \sum\limits_{G \in \cF^\sigma_{v(F_1)+v(F_2)-\ell}} p(F_1,F_2,G) \cdot G \quad\mbox{for every }F_1,F_2 \in \cF^\sigma \;,\]
which then uniquely extend to the whole set $\cA^\sigma$ by~\cite[Lemma 2.4]{Razborov07}. For brevity, we write $\cA$ to denote $\cA^\emptyset$. The \emph{averaging operator} $\llbracket \cdot \rrbracket_\sigma: \cA^\sigma \to \cA$ is defined by 
\[\llbracket F \rrbracket_\sigma = p_F^\sigma \cdot F^{\emptyset},\]
where $F^{\emptyset}$ is the unlabelled $3$-graph obtained from the $\sigma$-flag $F$ by forgetting about its embedding of $\sigma$, and $p_F^{\sigma}$ is the probability that a random injection from $V(\sigma)$ to $V(F^{\emptyset})$ yields an embedding of $\sigma$ such that the resulting $\sigma$-flag is isomorphic to $F$.

We now relate flag algebras to asymptotic properties of $3$-graphs. We say that a sequence of $K_4^-$-free $3$-graphs $(H_n)_{n\in\Nat}$ with $v(H_n)\rightarrow \infty$ is \emph{convergent} if  $\lim_{n\to\infty} p(F,H_n)$ exists for every $F\in\cF$. Since $p(F, H_n)\in[0,1]$ for every pair $(F, H_n)$, Tychonoff's theorem implies that every sequence $(H_n)_{n\in\Nat}$ has a convergent subsequence. For each convergent sequence of graphs $(H_n)_{n\in \Nat}$, we define a function $\phi: \cA \to \Real$, which we call the \emph{limit} of $(H_n)_{n\in \Nat}$, by letting 
\[\phi(F) := \lim_{n\rightarrow \infty} p(F,H_n) \quad\mbox{for every }F\in \cF,\]
and extending it to the rest of $\cA$ by linearity. For a given type $\sigma$, we define convergence and limits for $\sigma$-flag sequences analogously.

Let $\Hom^+(\cA,\Real)$  be the set of all algebra homomorphisms $\phi$ from $\cA$ to $\Real$ such that $\phi(F)\ge0$ for every $F\in\cF$. By construction, for every convergent sequence of $3$-graphs $(H_n)_{n\in \Nat}$ the associated limit $\phi$ is an element of $\Hom^+(\cA,\Real)$. Similarly, given a type $\sigma$ we let $\Hom^+(\cA^\sigma,\Real)$ denote the set of all algebra homomorphisms $\psi$ from $\cA^{\sigma}$ to $\Real$ with $\psi(F^{\sigma})\geq 0$ for every $F^{\sigma}\in \cF^{\sigma}$.

Observe that given a fixed embedding of $\sigma$ in a $3$-graph $H_n$, we have a map $\psi_n: \ F\mapsto p(F, H_n)$ sending a $\sigma$-flag $F$ to its $\sigma$-flag density in $H_n$, which extends by linearity to a map $\Real \cF^{\sigma} \to \Real$.  For a sequence of $3$-graphs $(H_n)_{n\in\Nat}$ converging to the limit $\phi \in \Hom^+(\cA,\Real)$ and $\sigma$ a type such that $\phi(\unlab\sigma\sigma) > 0$, we let $(\Pr_n^\sigma)$ be the sequence of probability distributions on such maps $\psi_n :\Real \cF^{\sigma} \to \Real$ after turning $H_n$ into a $\sigma$-flag by selecting an embedding of $\sigma$ in $H_n$ uniformly at random.  Similarly, let $\Pr_\phi^\sigma$ be the unique probability distribution on $\Hom^+(\cA^\sigma,\Real)$ satisfying
\begin{equation}\label{eq:hom_avg}
\E\limits_{\psi \sim \Pr_\phi^\sigma} \psi\left(f\right) = \frac{\phi\left(\unlab{f}\sigma\right)}{\phi\left(\unlab\sigma\sigma\right)} \quad {\forall f \in \cA^\sigma}\;.
\end{equation}
Razborov proved the existence and uniqueness of $\Pr_\phi^\sigma$ in~\cite[Theorem 3.5]{Razborov07}, and further that the sequence of probability distributions $({\Pr_n^\sigma})_{n \in \Nat}$ converges weakly to $\Pr_\phi^\sigma$ as $n \rightarrow \infty$ in~\cite[Theorem 3.12]{Razborov07}.

\subsection{Tight paths}\label{subsection: tightpaths}

\begin{figure}
\hfill\foreach \n in {12,13,14}{ \includegraphics[scale=0.70,page=\n]{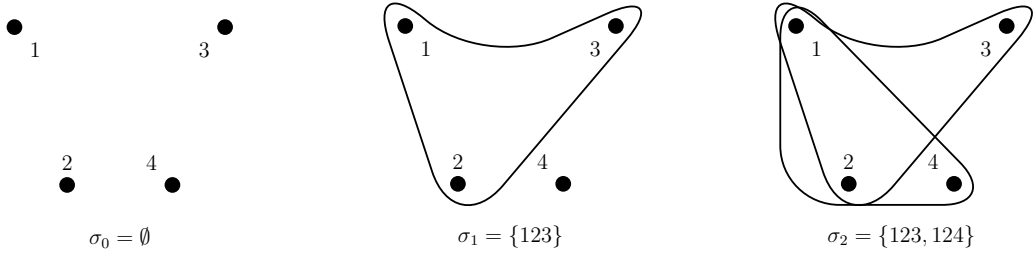}\hfill}

\caption{The $4$-vertex types $\sigma_0,\sigma_1$ and $\sigma_2$.}

\label{fig:sigmatypes}
\end{figure}
A crucial step in our proof of the stability result stated in Theorem~\ref{theorem: stability} will be constructing an auxiliary orientation of almost all pairs of the vertices of a $K_4^-$-free $3$-graph $H$ with large codegree.
In order to do so, we shall use a relation on the vertex-pairs of $H$ associated with what is known as \emph{tight connectivity}.

\begin{definition}\label{def: tightly connected}
We say that a pair of vertices $\{a,b\}$ is \emph{tightly connected} to a pair $\{c,d\}$ by a path of length $\ell$ if there exists a sequence of distinct $\ell+2$ vertices $v_1,v_2,\dots,v_{\ell+2}$ and $\ell$ edges $e_1,e_2,\dots,e_\ell$ such that $e_i = \{v_i,v_{i+1},v_{i+2}\}$ for every $i \in [\ell]$, $\{a,b\} \subseteq e_1$ and $\{c,d\} \subseteq e_\ell$.
\end{definition}
\noindent Having a fixed orientation of a given pair of vertices $\{a,b\}$ and a tight path $P$ of length $\ell$ connecting the pair $\{a,b\}$ to a pair $\{c,d\}$ allows us to propagate the orientation of $\{a,b\}$ to $\{c,d\}$ in a rather natural way:
there is a unique orientation of the remaining two pairs contained in the first edge $e_1$ of $P$ so that the three pairs in $e_1$ form a cyclically oriented triangle, which fixes the orientation of one pair contained in the next edge of $P$. Repeating the procedure $\ell$ times along the edges of $P$ yields an orientation for the pair $\{c,d\}$.

Let us now introduce a flag algebra notion, which we will use to capture tight connectivity in the proof of Theorem~\ref{theorem: stability}.
Fix $\sigma_0,\sigma_1$ and $\sigma_2$ to be the $4$-vertex types depicted in Figure~\ref{fig:sigmatypes}.
Note that, up to labelling of the vertices, $\sigma_0,\sigma_1$ and $\sigma_2$ are the only $K_4^-$-free $4$-vertex types. Let $\cP^{\sigma_0}_3$, $\cP^{\sigma_1}_3$ and $\cP^{\sigma_2}_3$ be the sets of $5$-vertex $\sigma_0$-, $\sigma_1$- and $\sigma_2$-flags depicted in  Figure~\ref{fig:pathS0}, Figure~\ref{fig:pathS1} and Figure~\ref{fig:pathS2} respectively.  A straightforward inspection of these flags yields the following.
\begin{observation}\label{observation: tight paths}
For any $i \in \{0,1,2\}$ and any $F^{\sigma_i} \in \cP^{\sigma_i}_3$, the pair $\{1,2\}$ is tightly connected to the pair $\{3,4\}$ in $F^{\sigma_i}$ by a path of length $3$. 
\end{observation}
\begin{figure}
\hfill\foreach \n in {1,2,3}{ \includegraphics[scale=0.70,page=\n]{fig}\hfill}\\
\hfill\foreach \n in {4,5,6}{ \includegraphics[scale=0.70,page=\n]{fig}\hfill}
\caption{The set of $\sigma_0$-flags $\cP^{\sigma_0}_3$.}
\label{fig:pathS0}
\end{figure}

\begin{figure}

\hfill\foreach \n in {7,8}{ \includegraphics[scale=0.70,page=\n]{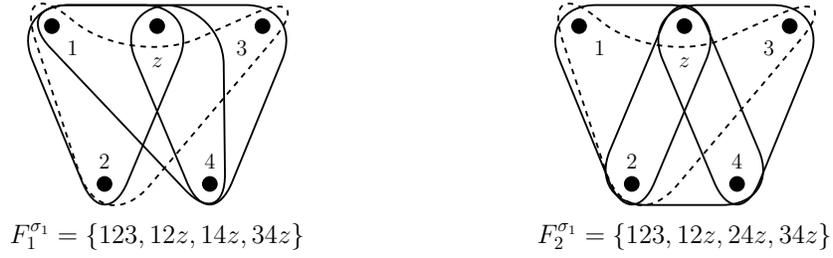}\hfill}
\caption{The set of $\sigma_1$-flags $\cP^{\sigma_1}_3$. The edge induced by the labelled vertices $1,2,3,4$ is represented by a dashed curve.}
\label{fig:pathS1}
\end{figure}

\begin{figure}

\hfill\foreach \n in {9,10}{ \includegraphics[scale=0.70,page=\n]{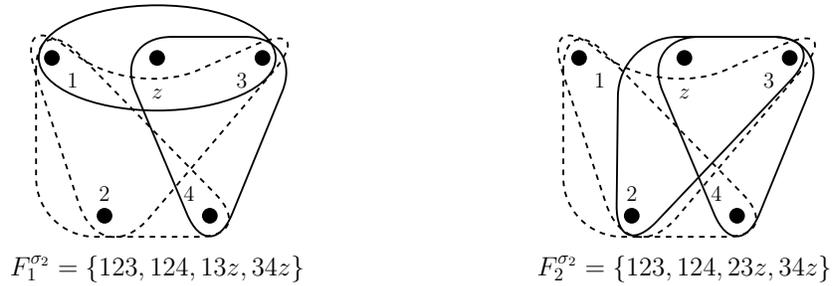}\hfill}
\caption{The set of $\sigma_2$-flags $\cP^{\sigma_2}_3$. The two edges induced by the labelled vertices $1,2,3,4$ are represented by dashed curves.}

\label{fig:pathS2}

\end{figure}

For $i \in \{0,1,2\}$, we define $\cP_i \in \cA$ to be the following non-labelled expression:
\[\cP_i := \unlab{\left( 16 \cdot \sum\limits_{F \in \cP^{\sigma_i}_3} F-\sum\limits_{F \in \cF^{\sigma_i}_5} F \right)^2}{\sigma_i}=\unlab{\left( 15 \cdot \sum\limits_{F \in \cP^{\sigma_i}_3} F \, -\sum\limits_{F \in \cF^{\sigma_i}_5 \setminus \cP^{\sigma_i}_3} F \right)^2}{\sigma_i}\;.\]
Note that $\cP_0$, $\cP_1$ and $\cP_2$ can be expressed as a linear combination of the elements of $\cF_6$. We now relate the expressions $\cP_i$ to the existence of (short) tight paths.

\begin{proposition}\label{proposition: tightly connected quadruples}
Let $(H_n)_{n\in\Nat}$ be a convergent sequence of $K_4^-$-free $3$-graphs and let $\phi \in \Hom^+(\cA,\Real)$ be its limit. If $\phi\left(\cP_0\right) = \phi\left(\cP_1\right) = \phi\left(\cP_2\right) = 0$ and $n\in\Nat$, then for all but $o\left(v(H_n)^4\right)$ quadruples of vertices $a,b,c,d \in V(H_n)$ the pair $\{a,b\}$ is tightly connected to $\{c,d\}$ by a path of length at most $3$.
\end{proposition}

\begin{proof}
Let $a,b,c,d\in V(H_n)$ be a quadruple of distinct vertices and $E'$ the set of edges in $H_n$ the quadruple induces. Since $H_n$ is $K_4^-$-free, there are, up to symmetry, exactly four possibilities:
\[\mbox{1) } E'=\emptyset   \,,
\quad\quad \mbox{2) } E'=\{abc\}     \,,
\quad\quad \mbox{3) } E'=\{abc,abd\} \,,\quad\mbox{ and } \quad \mbox{4) } E'=\{abc,bcd\} \,.
\]
In the last case, the edges $abc$ and $bcd$ already form a tight path of length $2$ between the pairs $\{a,b\}$ and $\{c,d\}$.

For the remaining three cases, our aim will be to find a suitable additional vertex $x \in V(H_n)$ such that the subgraph of $H_n$ induced by $\{a,b,c,d,x\}$ contains a tight path from $\{a,b\}$ to $\{c,d\}$ of length $3$. Based on $E'$, we break down the analysis into Claims~\ref{cl:1}, \ref{cl:2} and \ref{cl:3}. As each case can be treated in a very similar manner, we only deal with the case $E'=\emptyset$ in detail.

\begin{claim} \label{cl:1}
There are $o\left({v(H_n)}^4\right)$ choices of $a,b,c,d$ such that $\{a,b,c,d\}$ induces no edge and $\{a,b\}$ is not tightly connected to $\{c,d\}$ by a path of length $3$.
\end{claim}

\begin{proof}

Without loss of generality, we may assume $\phi\left(\unlab{\sigma_0}{\sigma_0}\right) > 0$ as otherwise the total number of choices of $\{a,b,c,d\}$ inducing no edge in $H_n$  is $o\left({v(H_n)}^4\right)$. Let ${\Pr}^{\sigma_0}_n$ be the probability distribution on mappings from $\Real \cF^{\sigma_0}$ to $\Real$ that arise from picking a copy of $\sigma_0$ in $H_n$ uniformly at random, i.e.,  choosing  a quadruple of vertices $Q$ uniformly at random conditioned on the event that $Q$ induces no edges in $H_n$.

Let $\psi_n$ be drawn according to the probability distribution ${\Pr}_n^{\sigma_0}$, i.e., fix a copy of $\sigma_0$ in $H_n$ uniformly at random, and let $a$, $b$, $c$ and $d$ be the vertices of $H_n$ corresponding to the $\sigma_0$ vertices $1$, $2$, $3$ and $4$, respectively. Recall that $\psi_n(F)$, where $F \in \cF_5^{\sigma_0}$, corresponds to the probability that a random vertex $x \in V(H_n) \setminus \{a,b,c,d\}$ extends $\{a,b,c,d\}$ into a $\sigma_0$-flag isomorphic to $F$. In~particular, $\sum_{F \in \cF_5^{\sigma_0}} \psi_n(F) = 1$.

Clearly, the pair $\{1,2\}$ is tightly connected to $\{3,4\}$ by a path of length $3$ in any $F \in \cP^{\sigma_0}_3$. In fact, a straightforward inspection reveals that the set $\cP^{\sigma_0}_3$ describes all the possibilities for such a path in a $K_4^-$-free setting. Therefore, the number of tight paths of length $3$ between $\{a,b\}$ and $\{c,d\}$ in $H_n$ is equal to $(v(H_n)-4) \cdot \sum_{F \in \cP_3^{\sigma_0}} \psi_n(F)$. Motivated by Construction~\ref{construction: C(T)}, we aim to prove the claim by showing that this expression is equal to $(1-o(1)) \cdot v(H_n)/16$ with probability $1-o(1)$.

Suppose to the contrary that there is $\eps_0 > 0$ and an infinite subsequence $(n_k)_{k\in \mathbb{N}}$ of $\mathbb{N}$ such that for every $k\in \mathbb{N}$ there are (at least) $\eps_0 \cdot v(H_{n_k})^4$ copies of $\sigma_0$ in $H_{n_k}$ such that the corresponding $\psi_{n_k}$ satisfies the two following equivalent inequalities
\[\left| \sum_{F \in \cP^{\sigma_0}_3} \psi_{n_k}(F) - \frac1{16} \right| \ge \eps_0
\quad \iff \quad 
\left( \sum_{F \in \cP^{\sigma_0}_3} \psi_{n_k}(F) - \frac1{16} \cdot \sum_{F \in \cF^{\sigma_0}_5} \psi_{n_k}(F) \right)^2 \ge \left(\eps_0\right)^2\;.\]
Fix such a $\psi_{n_k}$. The definition of the flag-algebra multiplication on $\cF^{\sigma_0}$ in turn yields that
\[\psi_{n_k}\left(\left( \sum_{F \in \cP^{\sigma_0}_3} F - \sum_{F \in \cF^{\sigma_0}_5} \frac{1}{16} \cdot F\right)^2 \right)\ge \left(\eps_0\right)^2 - o(1) \;.\]
As $\left({\Pr}^{\sigma_0}_{n_k}\right)_{k \in \Nat}$ weakly converges to the probability distribution ${\Pr}_\phi^{\sigma_0}$ associated with $\phi$, we have
\[\lim\limits_{k\to\infty} \E\limits_{\psi_{n_k} \sim {\Pr}_{n_k}^{\sigma_0}}\left[ \psi_{n_k}\left(\left( \sum_{F \in \cP^{\sigma_0}_3} F - \sum_{F \in \cF^{\sigma_0}_5} \frac{1}{16} \cdot F\right)^2 \right) \right]\,=\,
\E\limits_{\psi \sim {\Pr}_\phi^{\sigma_0}} \left[ \psi\left( \left( \sum\limits_{F \in \cP^{\sigma_0}_3} F-\sum\limits_{F \in \cF^{\sigma_0}_5} \frac1{16}\cdot F \right)^2 \right) \right] \;.
\]
Clearly, the limit on the left-hand side must be at least $\left(\eps_0\right)^3$, however, the expectation on the right-hand side is by~\eqref{eq:hom_avg} equal to $\frac1{256}\cdot\frac{\phi\left(\cP_0\right)}{\phi\left(\unlab{\sigma_0}{\sigma_0}\right)} = 0$, a contradiction.
\end{proof}

As ${\phi\left(\cP_1\right)}={\phi\left(\cP_2\right)}=0$, an analogous argument for embeddings of $\sigma_1$ and $\sigma_2$ in $(H_n)_{n \in \Nat}$ using the respective probability distributions $\left({\Pr}^{\sigma_1}_n\right) \to {\Pr}_\phi^{\sigma_1}$ and $\left({\Pr}^{\sigma_2}_n\right) \to {\Pr}_\phi^{\sigma_2}$ yields the following.

\begin{claim} \label{cl:2}
There are $o\left({v(H_n)}^4\right)$ choices of $a,b,c,d$ such that $\{a,b,c,d\}$ induces one edge $abc$ and $\{a,b\}$ is not tightly connected to $\{c,d\}$ by a path of length $3$.
\end{claim}

\begin{claim} \label{cl:3}
There are $o\left({v(H_n)}^4\right)$ choices of $a,b,c,d$ such that $\{a,b,c,d\}$ induces edges $\{abc,abd\}$ and $\{a,b\}$ is not tightly connected to $\{c,d\}$ by a path of length $3$.
\end{claim}
This concludes the proof of Proposition~\ref{proposition: tightly connected quadruples}.
\end{proof}

\subsection{Proof of $\pi_2(K_4^-) = 1/4$}\label{subsection: codegree density}
We shall prove Theorem~\ref{theorem: codegree density} by establishing a specific identity between subgraph densities of $7$-vertex $K_4^-$-free $3$-graphs that holds for any $\phi \in \Hom^+(\cA,\Real)$. This identity is obtained by an application of the semi-definite method of Razborov and some computer-aided flag algebra calculations. Before we can state this identity, we need to introduce a few more definitions.

For each $n\in \Nat$, let $T_n$ be a uniformly random tournament on $[n]$. It is straightforward to check that the sequence  $(C(T_n))_{n\in\Nat}$ is an almost surely convergent sequence of $3$-graphs.
Indeed, as $n\rightarrow \infty$, the density of any fixed $F$ in $C(T_n)$ almost surely converges to the probability that $F=C(T)$ when $T$ is a random tournament on $V(F)$ in which pairs are oriented randomly and independently of each other.
Let $\phi_T$ denote the limit of $(C(T_n))_{n\in \Nat}$ and let  \[\cE_T:=\{ F \in \cF: \phi_T(F) > 0\}.\]
We call the elements of $\cE_T$ \emph{tournament-realizable}.
The definition of $\cE_T$ readily yields the following.
\begin{observation}\label{observation: F in cE_T of phiT iff F is a tournament construction}
$F \in \cE_T$ if and only if there is a tournament $T_F$ on $V(F)$ such that $F=C(T_F)$.
\end{observation}

\begin{figure}
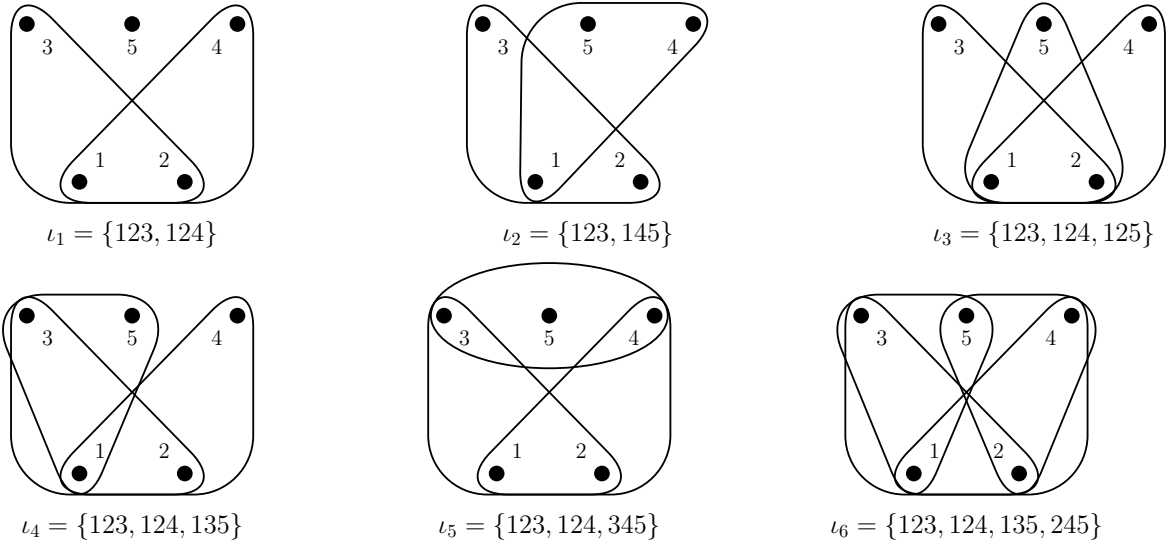

\foreach \n in {16,15,19}{ \includegraphics[scale=0.70,page=\n]{fig}\hfill}\\
\foreach \n in {18,17,20}{ \includegraphics[scale=0.70,page=\n]{fig}\hfill}
\caption{The $5$-vertex types $\iota_i$ for $i\in[6]$ used in the proof of Lemma~\ref{lemma: flag}.}

\label{fig:iotatypes}
\end{figure}
We now continue by introducing the remaining flag algebra notation.
Let $\cG := \cF \setminus \cE_T$ be the set of $K_4^-$-free $3$-graphs that are not tournament-realizable. For every $k \in \Nat$, we set $\cG_k := \cG \cap \cF_k$.
Let $\tau$ be the unique $2$-vertex type consisting of two labelled vertices $1$ and $2$.
Let $\mathrm N$ and $\mathrm E$ be the $3$-vertex $3$-graphs with $0$ and $1$ edge respectively, i.e.\  the (unlabelled) \emph{Non-edge} and \emph{Edge}, and similarly ${\mathrm N}^\tau$ and ${\mathrm E}^\tau$ the unique $3$-vertex $\tau$-flags with $0$ and $1$ edge respectively. A codegree assumption $\delta_2(H) \geq (c-o(1)) \cdot v(H)$ implies in the flag algebra language the statement that $\phi^{\tau} \left({\mathrm E}^{\tau}\right)\geq c$ with probability $1$, where  $\phi^{\tau}$ is drawn from ${\Pr}_\phi^\tau$. Equivalently, $\phi^\tau\left(\frac{1-c}{c} \cdot {\mathrm E}^{\tau} - {\mathrm N}^{\tau}\right) \geq 0$ with probability $1$.

 Let $\iota_1,\iota_2,\dots, \iota_6$ be the six $5$-vertex types depicted in Figure~\ref{fig:iotatypes}. For every $i \in [6]$, we let $k_i := \big|\cF^{\iota_i}_6\big|$, and fix an arbitrary enumeration $F^{\iota_i}_1, F^{\iota_i}_2,\dots, F^{\iota_i}_{k_i}$ of the elements of $\cF^{\iota_i}_6$. Note the particular enumeration that we will be using in the whole paper is the one used in our ancillary computer programs. Similarly, we fix an enumeration of $\cF^{\sigma_i}_5$ for $i \in \{0,1,2\}$.

For $i \in [6]$, we let $e_i$ be the $k_i$-dimensional vector from $\left(\cA^{\iota_i}\right)^{k_i}$ whose $j$-th coordinate is equal to $F^{\iota_i}_j$. By a straightforward computer search, we have $k_1 = 191$, $k_2 = 173$, $k_3 = 148$, $k_4 = 135$, $k_5 = 124$ and $k_6 = 95$ (this can be also checked using Flagmatic, see e.g.~\cite{FalgasRavryVaughan13,Vaughan13}).

Next, it is easy to check that $\phi_T\left(\unlab{\iota_2}{\iota_2}\right) = \phi_T\left(\unlab{\iota_5}{\iota_5}\right) = 0$. On the other hand, for $i \in \{1,3,4,6\}$ we have $\phi_T\left(\unlab{\iota_i}{\iota_i}\right) > 0$, and so the probability distribution ${\Pr}^{\iota_i}_{\phi_T}$ on $\Hom^+(\cA^{\iota_i}, \Real)$ is well-defined.
Note that in this case ${\Pr}^{\iota_i}_{\phi_T}$ is a purely atomic measure, and $\left|\supp\left({\Pr}^{\iota_i}_{\phi_T}\right)\right| \le 2^{10}$ (since the right-hand side is equal to the number of $3$-graphs on a labelled set of $5$ vertices).

For any proof of $\pi_2\left(K_4^-\right) \leq  1/4$, the support of ${\Pr}^{\iota_i}_{\phi_T}$ restricts arguments used in the proof that involve $\iota_i$-flags. In particular, if such a proof uses $\unlab{f^2}{\iota_i} \ge 0$ for some $f \in \cA^{\iota_i}$, then clearly we must have $\psi^{\iota_i}(f) = 0$ for all $\psi^{\iota_i} \in \supp\left({\Pr}^{\iota_i}_{\phi_T}\right)$.

Recall that $F^{\iota_i}_1, F^{\iota_i}_2, \dots, F^{\iota_i}_{k_i}$ is the fixed enumeration of $\cF^{\iota_i}_6$. We set $d_i$ to be the dimension of the subspace of $\Real^{k_i}$ generated by
\[\left\{ x \in \Real^{k_i} \cond \ \exists \psi^{\iota_i} \in \supp\left({\Pr}^{\iota_i}_{\phi_T}\right) \mbox{ such that }x_j = \psi^{\iota_i}\left(F^{\iota_i}_j\right) \mbox{ for all }j\in [k_i] \right\} \;.\]
By inspection of the tournaments whose cyclically oriented triangles give rise to a given type $\iota_i$ under the Erd{\H o}s--Hajnal construction, we have that $d_2=d_5=0$, $d_4=d_6=1$, $d_1=4$ and $d_3=6$.
Indeed, up to reversing the orientation of all the edges there are exactly $4$, $0$, $6$, $1$, $0$ and $1$ distinct tournaments on $[5]$ realising the types $\iota_1$, $\iota_2$, $\iota_3$, $\iota_4$, $\iota_5$ and $\iota_6$ respectively.
Each such a tournament $J$ corresponds to a different kind of embedding of $\iota_j$ in $\phi_T$, which we denote by $H^J$, and determines a different probability distribution on $\cF^{\iota_j}_{6}$ given by extending $H^J$ by a single random vertex of $\phi_T$.
By the definition of $\phi_T$, this is in a one-to-one correspondence with adding a new vertex $x$ to $J$ that has all its edges to $V(J)$ oriented randomly, and considering the random variable $C(J \cup \{x\})$.
Moreover, when viewing each such a probability distribution as a vector in $\Real\cF^{\iota_j}_{6}$, for every $j \in [6]$ the set of the $d_j$ vectors corresponding to all such distributions on $\cF^{\iota_j}_{6}$ is linearly independent.

Finally, yet another straightforward computer search yields $\left|\cF_6^\tau\right| = 1643$.
Moreover, exactly $167$ elements $F\in \cF_6^\tau$ are such that swapping the labels of $\tau$ (i.e., relabelling the vertex $1$ to $2$ and vice versa) yields a $\tau$-flag isomorphic to $F$.
Now observe that the expression $\unlab{F \times \left(3{\mathrm E}^{\tau} - {\mathrm N}^{\tau}\right)}{\tau}$ is invariant under such a swap of the labels of $\tau$ for every $F \in \cF^\tau$.
Therefore, there are exactly $(1643+167)/2=905$ pairwise different expressions of the form
$\unlab{F \times \left(3{\mathrm E}^{\tau} - {\mathrm N}^{\tau}\right)}{\tau}$ where $F\in \cF^\tau_6$.

Let $\cD := \big\{ \unlab{F \times \left(3{\mathrm E}^{\tau} - {\mathrm N}^{\tau}\right)}{\tau} | F \in \cF^\tau_6 \big\}$.
The discussion from the previous paragraph yields that $|\cD| = 905$.
Also note that $\phi(D) \ge 0$ for every $D\in \cD$ where $\phi$ is the limit of a sequence of $K_4^-$-free $3$-graphs $(H_n)_{n \in \Nat}$ with $\liminf_{n\to\infty} \frac{\delta_2(H_n)}{v(H_n)} \ge \frac14$, and
every $D \in \cD$ can be expressed as a linear combination of the elements of $\cF_7$.

We are now ready to state the main lemma of this section.
\begin{lemma}
\label{lemma: flag}
There exist
\begin{enumerate}
\item positive rationals $c_0, c_1$, $c_2$, $w_G$ for every $G \in \cG_7$, and $u_D$ for every $D \in \cD$,

\item rational matrices $I_i$ of sizes $\left(k_i - d_i \right) \times k_i$ for $i \in [6]$, and

\item positive definite rational matrices $Q_i \succ 0$ of sizes $\left(k_i - d_i\right) \times \left(k_i - d_i\right)$ for $i \in [6]$,
\end{enumerate}
such that the following identity holds in the theory of $K_4^-$-free $3$-graphs:
\begin{align}\label{equality: flag identity}
{\mathrm N} - 3 {\mathrm E}\,=\, \underbrace{\displaystyle\sum\limits_{D \in \cD} u_D \cdot D}_{\mathrm{CODEGREE}}
\,+\, \underbrace{\sum\limits_{i\in\{0,1,2\}} c_i \cdot \cP_i}_{\mathrm{TIGHT-PATH}}
\,+\, \underbrace{\sum\limits_{i\in[6]} \unlab{e^T_i I^T_i Q_i I_i e_i}{\iota_i}}_{\mathrm{POSITIVE-DEFINITE}}
\,+\, \underbrace{\sum\limits_{G \in \cG_7} w_G \cdot G}_{\mathrm{SLACK}}\;.
\end{align}
\end{lemma}

\begin{proof}
This is a standard flag algebra computation in the theory of $K_4^-$-free $3$-graphs, which was performed with the aid of a computer.
Files containing the rational matrices $I_1, \dots, I_6$ and $Q_1, \dots,Q_6$, the rationals $(u_D)$, $c_0$, $c_1$ and $c_2$, enumerations of $\mathcal{F}_7$ and
 $\mathcal{F}_6^{\iota_i}$ as well as verifications of the positive definiteness of the matrices $Q_i$ and the flag algebraic identity claimed by Lemma~\ref{lemma: flag} can be found at~\href{\webpage}{\webpage} as well as on the arXiv version of this paper; see the Appendix for details.
\end{proof}

\noindent With Lemma~\ref{lemma: flag} in hand, the remainder of the proof of Theorem~\ref{theorem: codegree density} is straightforward.

\begin{proof}[Proof of Theorem~\ref{theorem: codegree density}]
Suppose for a contradiction that $\pi_2(K_4^-) > 1/4$. Then there  exist $\eps_0 > 0$ and a sequence of $K_4^-$-free $3$-graphs $(H_n)_{n\in \Nat}$ with $v(H_n)\rightarrow \infty$ such that 
\[\liminf_{n\rightarrow\infty} \frac{\delta_2(H_n)}{v(H_n)} \geq 1/4 + \eps_0.\]
By compactness, there exists a convergent  subsequence $(H'_{n})_{n\in \Nat}$ of $(H_n)_{n\in\Nat}$. Let $\phi_0$ denote its limit. A double-counting argument yields that $\phi_0({\mathrm E}) = \phi_0\left(\unlab{{\mathrm E}^\tau}\tau\right) \ge 1/4 + \eps_0$. Therefore,
\begin{align}\label{inequality: bound on N-3E}
\phi_0({\mathrm N}-3{\mathrm E}) = \phi_0({\mathrm N}+{\mathrm E)}-4\cdot\phi_0({\mathrm E}) \le 1- 4\cdot \left(\frac{1}{4}+\eps_0\right) = -4\eps_0 < 0\;.
\end{align}
On the other hand, our codegree assumption yields that almost surely  $\phi^{\tau}_0(3{\mathrm E}^\tau - {\mathrm N}^\tau) \ge 4\eps_0$, where $\phi^{\tau}_0$ is a random homomorphism drawn from $\Pr^\tau_{\phi_0}$.
Therefore, for any non-negative $\left(u_D\right)_{D \in \cD}$,
\[\phi_0\left( \displaystyle\sum\limits_{D \in \cD} u_D \cdot D \right) = \displaystyle\sum\limits_{D \in \cD} u_D \cdot \phi_0\left(D \right) \ge 0.\]
As the evaluation of the remaining three summands in \eqref{equality: flag identity} is non-negative for any $\phi \in \Hom^+(\cA,\Real)$, Lemma~\ref{lemma: flag} yields that $\phi_0({\mathrm N}-3{\mathrm E}) \ge 0$ contradicting (\ref{inequality: bound on N-3E}).
\end{proof}

\section{Stability of Construction~\ref{construction: C(T)}}\label{section: stability}
In order to relate $K_4^-$-free $3$-graphs to Construction~\ref{construction: C(T)}, we establish the following proposition:
\begin{proposition}\label{prop:tournament}
Let $H=(V,E)$ be an $n$-vertex $3$-graph such that:
\begin{enumerate}[(i)]
\item every $7$-vertex subgraph of $H$ is tournament-realizable,

\item for all but $o(n^4)$ choices of $a,b,c,d\in V$ the pair $\{a,b\}$ is tightly connected to $\{c,d\}$ by a path of length at most $3$, and

\item all but $o(n^2)$ pairs of vertices of $H$ have codegree at least $n/4 - o(n)$.
\end{enumerate}
Then there is an oriented graph $\vec{G}$ on the vertex set $V$ such that:
\begin{enumerate}[(1)]

\item $\vec{G}$ has at least $n^2/2 - o(n^2)$ arcs,

\item if $u,v,w \in V$ span a transitive triangle in $\vec{G}$ then $uvw\notin E$, and

\item for all but $o(n^2)$ arcs $\vec{uv}$ in $\vec{G}$ there are at least $n/4 - o(n)$ vertices $w \in V$ such that $\{u,v,w\}$ induces a cyclically oriented triangle in $\vec{G}$.
\end{enumerate}
\end{proposition}

\begin{proof}
We shall use the assumption on the existence of tight paths of length at most $3$ in order to define an orientation for almost every pair of vertices $\{u,v\} \subseteq V^{(2)}$. For two vertices $u$ and $v$, we write $u \to v$ to denote the fact that we place an arc in $\vec{G}$ that goes from $u$ to $v$.

By Assumption (ii), there exists a pair $\{a,b\} \subseteq V$ such that $\{a,b\}$ is tightly connected to $\{c,d\}$ by a path of length at most $3$ for all but $o(n^2)$ pairs $\{c,d\} \subseteq V$. Fix such a pair $\{a,b\}$.

We start by placing an arc $a \to b$ in $\vec{G}$. Now, for every pair $\{c,d\}$ tightly connected to $\{a,b\}$ by a path of length at most $3$, we define the orientation of $\{c,d\}$ in $\vec{G}$ as follows: let $P_{cd}$ be an arbitrarily chosen tight path between $\{a,b\}$ and $\{c,d\}$ of length at most $3$, and let $S_{cd}$ denote the set of all pairs of vertices that are contained in some edge of $P_{cd}$. In other words, $S_{cd}$ is the $2$-shadow of $P_{cd}$.
As we have noted in Section~\ref{subsection: tightpaths}, there exists a unique orientation $O$ of the pairs in $S_{cd}$ such that $a\to b$, and every edge of $P_{cd}$ induces a cyclically oriented triangle. We orient the pair $\{c,d\}$ in $\vec{G}$ according to its orientation in $O$. Note that by our choice of $\{a,b\}$, the oriented graph resulting from this procedure will satisfy Property (1).

Let us first show that the orientation $\vec{G}$ we have just described is well-defined, i.e., that the orientation of a pair $\{c,d\}$ does not depend on the particular choice of the tight path $P_{cd}$. Indeed, given any two tight paths $P_{cd}$ and $P'_{cd}$ from $\{a,b\}$ to $\{c,d\}$ of length at most $3$, the subgraph induced by $V(P_{cd}) \cup V(P'_{cd})$ has at most $6$ vertices. Therefore, it is tournament-realizable by Assumption (i) on $H$. Let $S'_{cd}$ be the set of pairs of vertices contained in some edge of $E(P_{cd}) \cup E(P'_{cd})$. As in the previous paragraph, there exists a unique orientation of the pairs in $S'_{cd}$ such that $a\to b$ and every edge of $E(P_{cd}) \cup E(P'_{cd})$ induces a cyclically oriented triangle (the existence is guaranteed by the tournament-realizability, the uniqueness by the tightness of the paths). In particular, $P_{cd}$ and $P'_{cd}$ define the same orientation for $\{c,d\}$ in $\vec{G}$.

Next, we claim that if three vertices $u,v,w \in V$ span a transitive triangle in $\vec{G}$ then $uvw \notin E(H)$. The argument is very similar to the one in the previous paragraph. Without loss of generality, suppose the arcs on $\{u,v\}$ and $\{u,w\}$ are oriented as $u\to v$ and $u\to w$ in $\vec{G}$. Let $P_{uv}$ and $P_{uw}$ be (some) tight paths of length at most $3$ from $\{a,b\}$ to $\{u,v\}$ and $\{u,w\}$, respectively. Since the induced subgraph $F:=H[V(P_{uv}) \cup V(P_{uw})]$ has at most $7$ vertices, it is tournament-realizable by Assumption (i). In particular, there is a tournament $T_F$ with $a\to b$ realizing $F$. The presence of the tight paths implies that  $u \to v$ and $u \to w$ in $T_F$, so $uvw \notin E(F)$ and hence also $uvw \notin E(H)$.

It remains to establish Property (3). Fix $\eps_0 > 0$ and suppose for a contradiction there were at least $\eps_0 \cdot n^2$ arcs such that each of them is contained in fewer than $(1/4-\eps_0)n$ cyclically oriented triangles in $\vec{G}$. Since the arc density of $\vec{G}$ is $1-o(1)$, at least $\eps_0/2 \cdot n^2$ of these arcs must be contained in at least $(3/4+\eps_0/2)n$ transitive triangles in $\vec{G}$. By Property (2) of $\vec{G}$ that we have established in the previous paragraph, these transitive triangles correspond to non-edges of $H$. Thus we have found $(\eps_0/2) \cdot n^2$ pairs in $H$ with codegree at most $(1/4-\eps_0/2)n$ contradicting Assumption (ii).\end{proof}

\noindent The main result of this section is the following flag algebra version of Theorem~\ref{theorem: stability}.
\begin{theorem} \label{thm:flag}
Let $(H_n)_{n\in\Nat}$ be a  sequence of $K_4^-$-free $3$-graphs. If \[\liminf_{n\rightarrow \infty} \frac{\delta_2(H_n)}n = \pi_2(K_4^-) = \frac{1}{4},\] 
then $(H_n)_{n\in\Nat}$ converges to $\phi_T$.
\end{theorem}

\begin{proof}
Fix any sequence $(H_n)_{n \in \Nat}$ of $K_4^-$-free $3$-graphs with $\delta_2(H_n)\ge n/4-o(n)$. By compactness, pass to a subsequence $(H'_n)_{n \in \Nat}$ convergent to some limit $\phi$. Let $V_n$ denote the vertex-set of $H'_n$ and set $v_n := |V_n|$. Our aim is to use Proposition~\ref{prop:tournament} and assign to each element of the sequence a nearly-complete oriented graph. However, we will first `clean-up' the sequence $(H'_n)$ using standard regularity tools.

As established in the proof of Theorem~\ref{theorem: codegree density}, $\phi({\mathrm N}-3{\mathrm E})\geq 0$. On the other hand, our codegree assumption yields $\phi({\mathrm N}-3{\mathrm E})\leq 0$, and hence $\phi({\mathrm N}-3{\mathrm E})=0$. Now Properties 1--4 from Lemma ~\ref{lemma: flag} guarantee that the $\phi$-evaluation of each of the four summands in~\eqref{equality: flag identity} is non-negative. In particular, $\phi({\mathrm N}-3{\mathrm E})=0$ and Identity~\eqref{equality: flag identity} together imply that  the $\phi$-evaluation of all four summands on the right hand side of~\eqref{equality: flag identity} must be equal to $0$.

Recall that $\cG_7$ is the set of all $7$-vertex $3$-graphs that are not tournament-realizable. Since $w_G > 0$ for every $G \in \cG_7$ (Property 1 in Lemma~\ref{lemma: flag}), the ``SLACK'' summand in \eqref{equality: flag identity} evaluating to zero under $\phi$ implies that $\phi(G) = 0$ for every $G\in \cG_7$. Therefore, there are at most $o\left({v_n}^7\right)$ induced $7$-vertex subgraphs in $H'_n$ that are not tournament-realizable. Applying the induced version of the Hypergraph Removal Lemma of R\"odl and Schacht~\cite[Theorem 6]{RodlSchacht09}, we can add or remove $o({v_n}^3)$ edges in $H'_n$ for each $n\in \Nat$ in order to obtain a modified sequence of $3$-graphs $(H''_n)_{n\in\Nat}$ such that in fact every induced $7$-vertex subgraph of $H''_n$ is tournament-realizable. (Note this implies in particular that $H''_n$ remains $K_4^-$-free for every $n \in \Nat$.)

Since $\left|E(H''_n) \Delta E(H'_n)\right| = o({v_n}^3)$, the sequence $(H''_n)_{n\in\Nat}$ is also convergent and its limit is $\phi$. Thus it is enough to show that the limit of $(H''_n)_{n \in \Nat}$ is equal to $\phi_T$. Also, by construction, the codegree assumption on $H'_n$ and averaging imply that all but $o({v_n}^2)$ pairs of vertices in $H''_n$ have codegree at least $v_n/4 - o(v_n)$.

As we noted above, the ``TIGHT-PATH'' summand in~\eqref{equality: flag identity} must $\phi$-evaluate to zero. Since $c_i > 0$ for every $i \in \{0,1,2\}$ (Property 1 in Lemma~\ref{lemma: flag}), this implies that $\phi(\cP_0) = \phi(\cP_1) = \phi(\cP_2) = 0$. Proposition~\ref{proposition: tightly connected quadruples} then yields that $\{a,b\}$ is tightly connected to $\{c,d\}$ by a path of length at most $3$ for all but $o(v_n^4)$ choices of $a,b,c,d$.

We have thus verified in the three paragraphs above that the elements of $(H''_n)_{n\in \Nat}$ satisfy Conditions (i)--(iii) of Proposition~\ref{prop:tournament}. Applying it, we assign to each $3$-graph $H''_n$ an oriented graph ${\vec{G}_n}$ on $V_n$ satisfying Properties (1)--(3) from Proposition~\ref{prop:tournament}.

Now, every triple of vertices $\{u,v,w\}$ inducing an edge in $H''_n$ either corresponds to a cyclically oriented triangle of ${\vec{G}_n}$, or contains one of the $o({v_n}^2)$ pairs that do not span an arc in ${\vec{G}}_n$. Since all but $o({v_n}^2)$ pairs from $V(H''_n)$ have codegree $v_n/4 + o(v_n)$, the density of cyclically oriented triangles in ${\vec{G}}_n$ must be at least $1/4 - o(1)$.

On the other hand, a simple application of the Cauchy--Schwarz Inequality (see~\eqref{equation: bound on C(T)} below) shows that the density of cyclically oriented triangles in any oriented graph is at most $1/4+o(1)$. Therefore,  for all but $o({v_n}^3)$ triples $\{u,v,w\} \subseteq V(H''_n)$,  $\{u,v,w\}$ induces a cyclically oriented triangle in ${\vec{G}_n}$ if and only if it induces an edge in $H''_n$.

Let $T_n$ be obtained from $\vec{G}_n$ by adding the $o({v_n}^2)$ missing arcs with arbitrary orientations. Clearly, Property (3) of ${\vec{G}}_n$ from Proposition~\ref{prop:tournament} does transfer to $T_n$, and hence  $(T_n)_{n\in \Nat}$ is a sequence of $o(1)$-quasirandom tournaments by a result of Coregliano and Razborov~\cite[Theorem~3.2]{CoreglianoRazborov17}.
A classical result of Chung and Graham~\cite[Theorem~1 (P1)]{ChungGraham91} states that for any fixed tournament $S$, the proportion of $v(S)$-vertex subsets of $V_n$ inducing in $T_n$ a copy of $S$ is asymptotically equal to the expected density of $S$ in a random tournament. In particular, for any $3$-graph $F$, the proportion of $v(F)$-vertex subsets of $V_n$ inducing in $C(T_n)$ a copy of $F$ is asymptotically $\phi_T(F)$. In other words, $\phi_T$ is the limit of $(C(T_n))_{n\in \Nat}$. As $\vec{G}_n$ differs from $T_n$ on $o(v_n^2)$ arcs, we have $\left|E(C(T_n)) \Delta E(H''_n)\right| = o(v_n^3)$. Hence, $\phi = \phi_T$. It follows that $\phi_T$ is the only possible accumulation point of the original sequence $(H_n)_{n\in \mathbb{N}}$, so (by compactness again) the sequence $(H_n)_{n\in \mathbb{N}}$ converges to $\phi_T$.
\end{proof}

\section{Tournaments with large codegree and Hadamard matrices}\label{section: Tournament vs Hadamard}
In this section, we relate tournaments $T$ that have $\delta_2(T)$ close to $v(T)/4$ to skew Hadamard matrices. Recall from the introduction that a \emph{skew Hadamard matrix} of order $n$ is an $n\times n$ square matrix $A$ with $\pm1$ entries such that (i) $AA^t= nI_n$, and (ii) $A+A^t=2I_n$. Here $I_n$ denotes the $n\times n$ identity matrix and $A^t$ denotes the transpose of the matrix $A$.

Let $t(n)$ be the largest minimum codegree in an $n$-vertex tournament, i.e.,  
\[t(n):=\max\{\delta_2(T):  \ T \textrm{ a tournament on }[n]\}.\]
\begin{proposition}\label{proposition: upper bound on t(n)}
$t(n)\leq \floor{\frac{n+1}{4}}$ for every $n \in \Nat$.
\end{proposition}
\begin{proof}
Let $n=4k+r$ with $r\in\{0,1,2,3\}$, and let $T$ be a tournament on $[n]$. Set $u:=\floor{\frac{n-1}{2}}\cdot \ceil{\frac{n-1}{2}}$. Double-counting and applying the Cauchy--Schwarz Inequality, we have that the number $\vert C(T)\vert$ of cyclically oriented triangles in $T$ satisfies:
\begin{align}\label{equation: bound on C(T)}
	2\ \vert C(T)\vert+\binom{n}{3} &=\sum_{x\in [n]} d^-(x)d^+(x)\leq nu.
\end{align}
Since $\binom{n}2 \cdot \delta_2(T) \le 3\ \vert C(T)\vert$, rearranging the terms in (\ref{equation: bound on C(T)}) yields the following estimate
\begin{align}\label{equation: bound on delta2 of C(T)}
\delta_2(T)&\leq \left\lfloor \frac{3\ \vert C(T)\vert}{\binom{n}{2}} \right\rfloor \leq \left\lfloor\frac{3u}{n-1}-\frac{n-2}{2}\right\rfloor.
\end{align}
The right-hand side of (\ref{equation: bound on delta2 of C(T)}) is equal to $k+1$ if $r=3$, and $k$ otherwise.
\end{proof}
\noindent The proof above implies that for $n\equiv 3\  (\mathrm{mod}\ 4)$, the upper bound of $\frac{n+1}{4}$ on $\delta_2(T)$ can only be attained if $T$ is highly regular. Explicitly, we have the following corollary.
\begin{corollary}\label{corollary: extremals are regular}
If $T$ is an $n$-vertex tournament with $\delta_2(T) = \frac{n+1}4$, then $d^-(v)=d^+(v)=\frac{n-1}2$, $C_T(v,w)=\frac{n+1}4$ and $R_T(v,w)=\frac{n-3}4$  for every distinct $v,w\in V(T)$.
\end{corollary}

\begin{proof}
Since $|C(T)|$ meets the upper bound form Proposition~\ref{proposition: upper bound on t(n)}, we must have equality in both estimates (\ref{equation: bound on C(T)}) and (\ref{equation: bound on delta2 of C(T)}). Note that $3u/(n-1)-(n-2)/2$ is exactly $k+1$, without rounding. In particular, every vertex has both in-degree and out-degree exactly $\frac{n-1}2$, and every pair of vertices has codegree exactly $\frac{n+1}4$. Fix a pair $\{v,w\}\in V(T)^{(2)}$, and assume without loss of generality it is oriented as $\vec{vw}$ in $T$. Since $d^-(v) + d^+(w) = n-1$, we have by the inclusion-exclusion principle that
\[R_T(v,w) =(n-2) - d^-(v)-d^+(w) +C_T(v,w)= (n-2) - (n-1) + \frac{n+1}{4} = \frac{n-3}{4}.\]
\end{proof}
When is the upper bound in Proposition~\ref{proposition: upper bound on t(n)} tight? For $n\equiv 3 \pmod 4$, this question is very closely related to the existence of a skew Hadamard matrix of order $n+1$. As we have mentioned in the introduction, skew Hadamard matrices are known to exist for infinitely many orders, and conjectured to exist for any order divisible by four.
\begin{proposition}\label{proposition: t(n) bound tight iff exists skew Hadamard matrix}
For $n\equiv 3 \pmod 4$, $t(n)=\frac{n+1}{4}$ if and only if there exists a skew Hadamard matrix of order $n+1$.
\end{proposition}

\begin{proof}
Fix $n=4k+3$ and $T$ a tournament on $[n]$ with $\delta_2(T)=\lfloor \frac{n+1}{4}\rfloor$. By Corollary~\ref{corollary: extremals are regular}, all in- and out-degrees in $T$ are equal to $2k+1$ and all codegrees are equal to $k+1$. An old result of Reid and Brown~\cite[Theorems 1 and 2]{ReidBrown72} yields that such a tournament exists if and only if there exists a skew Hadamard matrix of order $4k+4$. For the sake of completeness, we sketch how one can obtain a skew Hadamard matrix from such a tournament (i.e.\  the ``only if'' direction). Note that all the steps can be also applied in the reverse order to establish the ``if'' direction.

Let $A$ be the adjacency matrix of $T$ with $A_{ii}=0$, $A_{ij}=+1$ if $\vec{ij}\in T$ and $A_{ij}=-1$ if $\vec{ji}\in T$. Now fix a pair $\{i,j\}\in [n]^{(2)}$ with $i<j$ and, say, $\vec{ji}\in E(T)$. Consider the $2\times(n-2)$-submatrix $M:=A(\{i,j\},[n]-\{i,j\})$. By Corollary~\ref{corollary: extremals are regular}, this submatrix has exactly $k+1$ columns $(1,-1)^T$, corresponding to the exactly $k+1$ vertices $v$ with $\vec{iv},\vec{vj}\in E(T)$. The row of $M$ indexed by $i$ has a further $2k+1-(k+1)=k$ entries equal to $1$, so the column $(1,1)^T$ appears in $M$ exactly $k$ times. Likewise, the column $(-1,-1)^T$ appears exactly $k$ times in $M$, and finally the column  $(-1,+1)^T$  appears $R_T(i,j) =k$ times. Thus, the scalar product of the rows $i$ and $j$ in $A$ is equal to $-(k+1)+k+k-k=-1$.

Now append to $A$ a row indexed by $0$ with all entries set to $1$, and then a column indexed by $0$ whose first entry is $0$ and all of whose other entries are set to $-1$. This yields an $(n+1)\times (n+1)$ matrix $C$ satisfying $C=-C^t$ (so $C$ is skew-symmetric) and $CC^t=nI_{n+1}$ (for the latter, observe that any pair of distinct rows is orthogonal). 
Consider now the $(n+1) \times (n+1)$ matrix $H:= C+I_{n+1}$. Then $H$ is a square matrix with $\pm 1$ entries such that $H+H^t=C+C^t+2I_{n+1}=2I_{n+1}$ and \begin{align*} HH^t=(C+I_{n+1})(C^t+I_{n+1})=CC^t+I_{n+1} +C+C^t= (n+1)I_{n+1}\,. \end{align*}
Thus $H$ is a skew Hadamard matrix of order $n+1$.
\end{proof}
\noindent Clearly, knowing the exact value of $t(4k+3)$ for some $k\in \Nat$ allows us to also give good bounds on $t(n)$ for $n$ close to $4k+3$. We summarize our knowledge of $t(n)$ in the following proposition.

\begin{proposition}\label{prop:  existence of good tournaments}
Fix a non-negative integer $k$. If there exists a skew Hadamard matrix of order $4k+4$, then
\begin{multicols}{2} \begin{enumerate}[\quad 1)]
  \item $k+1 \ge t(4k+4) \ge k$,
  \item $t(4k+3)=k+1$,
  \item $t(4k+2)=k$, and
  \item $k \ge t(4k+1) \ge k-1$.
\end{enumerate} \end{multicols}
\end{proposition}
\begin{proof}
 All four upper bounds on $t(n)$ follow from Proposition~\ref{proposition: upper bound on t(n)} so we only need to establish the lower bounds.

Suppose there exists a skew Hadamard matrix of order $4k+3$. Proposition~\ref{proposition: t(n) bound tight iff exists skew Hadamard matrix} establishes the lower bound in 2). Let $T_k$ be an optimal tournament on $4k+3$ vertices. Deleting an arbitrary vertex from $T_k$ yields a $(4k+2)$-vertex tournament $T^-_k$ satisfying $\delta_2(T^-_k) \geq k$, proving 3). Similarly, deleting two arbitrary vertices from $T_k$ yields a tournament $T_k^=$ on $4k+1$ vertices with $\delta_2(T^=_k)\geq k-1$, proving 4).

It remains only to show that $t(4k+4) \ge k$. Fix an arbitrary vertex $x\in V(T_k)$, and let $T_k^+$ be a $(4k+4)$-vertex tournament constructed from $T_k$ in the following way: add a new vertex $y$ to $T_k$, orient the arc $xy$ arbitrarily, and give each pair $vy$ with $v \in V(T_k)\setminus \{x\}$ the orientation opposite to the orientation of $vx$ (i.e., if $\vec{vx}\in T_k$ then $\vec{yv}\in T_k^+$, and if $\vec{xv}\in T_k$ then $\vec{vy}\in T_k^+$).

We claim that  $\delta_2(T_k^+)\geq k$. By the assumption on $T_k$, it is enough to check that all the pairs containing $y$ are in at least $k$ cyclically oriented triangles. First, consider a pair $vy$ for some $v\in V(T_k)\setminus\{x\}$. Since $vy$ is oriented in the opposite direction to $vx$, there are $R_{T_k}(v,x)=k$ cyclically oriented triangles in $T_k^+$ containing $v$ and $y$, as desired.  Finally, consider the pair $xy$. Without loss of generality, we oriented it as $\vec{xy}$, so every one of the $2k+1$ in-neighbours of $x$ is an out-neighbour of $y$. Therefore, $C_{T_k^+}(x,y) = 2k+1>k$, which concludes the proof.
\end{proof}

\section{The codegree threshold of $K_4^-$}\label{section: codegree threshold}
We now turn our attention to determining the exact value of $\ex_2\left(n, K_4^-\right)$ for infinitely many $n$. The main result of this section is the following.

\begin{theorem}\label{theorem: largest size among codegree extremal achieved by tournaments}
Let $G$ be a $K_4^-$-free $3$-graph on $n$ vertices and with $\delta_2(G)=n/4-o(n)$.  Then there is a tournament $T$ on $V(G)$ such that $|G|\le |C(T)|$. Moreover, if $|G| = |C(T)|$ then $G=C(T)$.
\end{theorem}
\noindent Before proving Theorem~\ref{theorem: largest size among codegree extremal achieved by tournaments}, let us apply it to determine the codegree threshold of $K_4^-$.

\begin{proof}[Proof of Theorem~\ref{theorem: codegree threshold}]
Theorem~\ref{theorem: largest size among codegree extremal achieved by tournaments} implies that tournament constructions have the largest number of edges over all near-extremal constructions. In particular if $G$ is a $K_4^-$-free $3$-graph on $n$ vertices and $n$ is sufficiently large, then there exists a tournament $T$ on $n$ vertices such that 
\begin{align}\label{eq: bound on delta2, proof of codeg threshold}\delta_2(G) \leq \floor{\frac{3\vert G\vert }{\binom{n}{2}}} \leq \floor{\frac{3\vert C(T)\vert}{\binom{n}{2}}} \leq \floor{\frac{n+1}{4}}\,,\end{align}
where the last inequality follows from~\eqref{equation: bound on delta2 of C(T)}. This gives the claimed upper bound on $\ex_2(n, K_4^-)$. Proposition~\ref{prop:  existence of good tournaments} implies we have equality for $n=4k+3$ and $n=4k+2$ if there exists a skew Hadamard matrix of order $4k+4$.

Conversely, note that by Proposition~\ref{proposition: t(n) bound tight iff exists skew Hadamard matrix}, for $n=4k+3$ the last inequality in~\eqref{eq: bound on delta2, proof of codeg threshold} is an equality if and only if there exists a skew Hadamard matrix of order $4k+4$ and $T$ is an $n$-vertex tournament with $\delta_2(T)=t(n)$. Furthermore, by  the codegree-regularity for such tournaments observed in Corollary~\ref{corollary: extremals are regular}, we have $\floor{\frac{3\vert G\vert }{\binom{n}{2}}} = \floor{\frac{3\vert C(T)\vert}{\binom{n}{2}}}$ only if $\vert G\vert = \vert C(T)\vert$. By the ``moreover'' part of Theorem~\ref{theorem: codegree threshold} and~\eqref{eq: bound on delta2, proof of codeg threshold}, it follows that $\delta_2(G)=\frac{n+1}{4}$ only if $G=C(T)$. In particular, for $n=4k+3$ sufficiently large, if $\ex_2(n, K_4^-)=\frac{n+1}{4}$ then every codegree-extremal $3$-graph for that value of $n$ is given by an Erd{\H o}s--Hajnal tournament construction, and there exists skew Hadamard matrices of order $n+1$.
\end{proof}

\subsection{Proof of Theorem~\ref{theorem: largest size among codegree extremal achieved by tournaments}}\label{section: proofexact}
Throughout  the remainder of this section, we assume $n_0 \in \Nat$ is sufficiently large, and $n \ge n_0$. Let $G$ be a $K_4^-$-free $3$-graph on $[n]$ with $\delta_2(G) = \frac{n}{4} - o(n)$, and $T$  a tournament on $[n]$ that maximises $\vert G\cap C(T)\vert$ (or, equivalently, minimises $\vert G\setminus C(T)\vert$). The definition of $T$ implies, amongst other things, that we cannot increase the size of the intersection $\vert G\cap C(T)\vert$ by reversing the orientation of an arc $\vec{xy}$ in $T$. In terms of the joint neighbourhoods of $\{x,y\}$, this means
\begin{align}\label{eq:Reverse}
 |(G\cap C(T))_{xy}|\ge |G_{xy}\cap \{z:\  \vec{xz},\vec{zy}\in T\}|.
\end{align}
We begin our proof by showing that the symmetric difference of $G$ and $C(T)$ is small. 
\begin{claim}\label{claim: T is close to G}
$\vert G \Delta C(T)\vert = o(n^3)$.
\end{claim}
\begin{proof}	
By Theorem~\ref{theorem: stability}, there exists a tournament $T'$ such that 
\begin{align*}
\vert G\setminus C(T)\vert\leq \vert G\setminus C(T')\vert \leq \vert G\Delta C(T')\vert = o(n^3).
\end{align*}
The second of the two inequalities in~\eqref{equation: bound on delta2 of C(T)} shows that $C(T)$ and $C(T')$, being tournament constructions, can have at most $\binom{n}{3}/4+O(n^2)$ edges each.
On the other hand, the codegree condition on $G$ tells us that $G$ must have at least $\binom{n}{2}\cdot\frac{n/4-o(n)}3 = \binom{n}{3}/4 - o(n^3)$ edges. Therefore, when $n$ is sufficiently large, we have
\[\vert C(T)\setminus G \vert = \vert C(T)\vert  - \vert G \vert +\vert G\setminus C(T)\vert = o(n^3),\]
and hence also $\vert G \Delta C(T)\vert = o(n^3)$.
\end{proof}
\noindent Next, we show that $T$ must be a quasirandom tournament.
\begin{claim}\label{claim: T is qr}
$T$ is $o(1)$-quasirandom.
\end{claim}

\begin{proof}
Since $\vert G \Delta C(T)\vert = o(n^3)$, $|G|={n \choose 3}/4 + o(n^3)$ and $\delta_2(G)=n/4 - o(n)$, there are only $o(n^2)$ pairs of vertices $x$ and $y$ that have $\vert n/4 - C_T(x,y) \vert = \Omega(n)$. Therefore, $T$ is $o(1)$-quasirandom by~\cite[Theorem~3.2]{CoreglianoRazborov17}.
\end{proof}
\noindent Let $B:=G\setminus C(T)$, and $M:=C(T)\setminus G$ be the $3$-graphs consisting of the \emph{bad} and \emph{missing} triples, respectively. Our aim is to show that $|M| \ge \vert B\vert $. Before we can do so, we need to prove some auxiliary results on the degrees and codegrees in $B$ and $M$.

Fix an arbitrary vertex $x\in V(G)$, and partition the arcs in $T$ corresponding to the pairs from $B_x$ into three oriented graphs as follows (see also Figure~\ref{fig:B- B+ B-+}):
\begin{figure}

\hfill\foreach \n in {23,24,25}{ \includegraphics[page=\n]{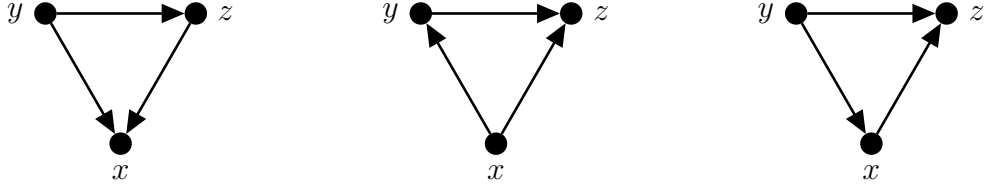}\hfill}

\caption{The arcs $\vec{yz}$ in the sets $B_x^-, B_x^+$ and $B_x^\mp$.}

\label{fig:B- B+ B-+}

\end{figure}
\begin{eqnarray*}
	B_x^-&:=&\{\vec{yz}\in T:\  xyz\in G,\ \{y,z\}\subseteq N_T^-(x) \},\\
	B_x^+&:=& \{\vec{yz}\in T:\  xyz\in G,\ \{y,z\}\subseteq N_T^+(x)\},\hbox{ and}\\
	B_x^{\mp} &:= & \{\vec{yz}\in T: \  xyz \in G,\ y \in N_T^-(x),\ z\in N_T^+(x)\}.
\end{eqnarray*}
\noindent Our aim in the next few claims is to show that all these three oriented graphs contain only $o(n^2)$ arcs, which (since $x$ is arbitrary) will imply that all links graphs in $B$ are sparse (Corollary~\ref{corollary:DeltaG-CT}). Let $X^-$ be the collection of those vertices $y\in N_T^-(x)$ that have out-degree $\Omega(n)$ in $B_x^-$.
\begin{claim}\label{claim:  Yx- small}
$\vert X^-\vert = o(n)$.
\end{claim}

\begin{proof}
Fix $y\in X^-$ and let $Y^+:=N^+_{B_x^-}(y)$ denote the out-neighbourhood of $y$ in $B_x^-$. We now consider the effect of reversing the orientation of $\vec{yx}\in T$. Let $T^R$ be the resulting tournament.
For each $w\in Y^+$, the triple $xyw$ now belongs to $G\cap C(T^R)$, hence $\vert (G\cap C(T))_{xy} \vert \ge \vert Y^+\vert = \Omega(n)$ by~\eqref{eq:Reverse}.

Clearly, the joint neighbourhood $(G\cap C(T))_{xy}$ is a subset of $N_T^+(x) \cap N_T^-(y)$. Also note that $Y^+ \subseteq N^-_T(x) \cap N^+_T(y)$.
Since both $Y^+$ and $(G\cap C(T))_{xy}$ have sizes $\Omega(n)$, the $o(1)$-quasirandomness of $T$ (see~\cite[Theorem~1 (P8)]{ChungGraham91}) yields that
\[\frac12 \cdot \vert Y^+\vert \cdot \vert (G\cap C(T))_{xy} \vert  - o(n^2) = \Omega(n^2)\]
pairs $(w, z)\in Y^+ \times (G\cap C(T))_{xy}$ are oriented as $\vec{wz}$ in $T$.

For each such a pair, we have that $ywz \in C(T)$. On the other hand, $\vec{yw} \in B^-_x$ and thus $xyw \in G$. Since $xyz \in G$ and $G$ is $K_4^-$-free, we have $ywz \in C(T)\setminus G = M$. We conclude that each vertex $y\in X^-$ is incident to at least $\Omega(n^2)$ missing triples of this form. Summing over all $y\in X^-$, we have
\[ \vert X^- \vert \times \Omega(n^2) = \left\vert\{ywz \in M: xyz \in G \cap C(T),\ \vec{yw}\in B_x^-, \ y \in X^-\}\right\vert  \leq \vert G\Delta C(T)\vert = o(n^3),\]
which, after division by $\Omega(n^2)$, yields the claimed bound $\vert X^- \vert = o(n)$.
\end{proof}
A symmetric argument yields that $|X^+|=o(n)$, where $X^+$ is the set of vertices $z\in N_T^+(x)$ with in-degree at least $\Omega(n)$ in $B_x^+$.

By definition, every vertex outside of $X^-\cup X^+$ has out-degree $o(n)$ in $B_x^-$ and in-degree $o(n)$ in $B_x^+$. Therefore, the number of arcs in $B^-_x$ and in $B^+_x$ is $o(n^2)$.
\begin{corollary}\label{corollary: bx- and bx+ are small}
$\vert B_x^-\vert = o(n^2)$ and $\vert B_x^+\vert = o(n^2)$.\qed
\end{corollary}
\noindent We now turn out attention to the cross-arcs $B_x^{\mp}$. We begin with the following simple claim: 
\begin{claim}\label{claim: bad cross edges imply many missing edges}
For every $\vec{yz} \in B_x^{\mp} $ and every $w$ such that $\vec{zw}, \vec{wy}\in T$,  at least one of the triples $wyz, wxy, wxz$ is in $M$. 
\end{claim}
\begin{proof}
Clearly, $wyz \in C(T)$. If $\vec{xw}\in T$ then also $wxy \in C(T)$, while if $\vec{wx} \in T$ then $wxz \in C(T)$.
Either way, the set $xyzw$ induces two triples in $C(T)$.
Since $xyz\in  G\setminus C(T)$ and $G$ is $K_4^-$-free, at least one of $wyz, wxy, wxz$ must be in $C(T) \setminus G = M$.
\end{proof}
Let $X_2^-$ be the collection of $y\in N_T^-(x)$ having a linear out-degree in $B_x^{\mp}$, i.e., $d^+_{B_x^{\mp}}(y) = \Omega(n)$. We shall focus on `typical' elements of $X_2^-$ by moving to a subset $Y\subseteq X_2^-$, where $y\in Y$ if and only if the following three properties are satisfied:
\begin{enumerate}[(i)]
\item $y$ is incident with $o(n)$ arcs in $B_x^- \cup B_x^+$,

\item $\vert d^+(y, N_T^+(x))- d^-(y, N_T^+(x))\vert = o(n)$, and

\item  $\vert M_y\vert = o(n^2)$.
\end{enumerate}
\noindent First of all, we show that the set $Y$ is not much smaller than $X_2^-$.
\begin{claim}\label{claim: Y and X2- have same size}
	$\vert Y\vert = \vert X_2^- \vert - o(n)$.
\end{claim}

\begin{proof}
By Corollary~\ref{corollary: bx- and bx+ are small}, only $o(n)$ vertices of $X_2^-$ can be incident to linearly many arcs in $B_x^- \cup B_x^+$ (in fact, the vertices in $X_2^-$ are by definition incident to no arc in $B_x^+$). Further, the $o(1)$-quasirandomness of $T$ yields that only $o(n)$ of vertices of $X_2^-$ can fail to satisfy Property (ii). Finally, $|M|=o(n^3)$ by Claim~\ref{claim: T is close to G}, hence there can be only $o(n)$ vertices $y \in [n]$ with a quadratic degree in~$M$, i.e., having $|M_y| = \Omega(n^2)$.
\end{proof}
For each $y\in Y$, let $Z_{xy}$ be the collection of vertices $z\in N_T^+(x)$ such that $\vec{yz}\in B^{\mp}(x)$.
Note that $|Z_{xy}| = \Omega(n)$ by definition of $X_2^-\supseteq Y$.
Next, we define $W_{xy}$ to be the collection of $w\in N_T^+(x)$ such that $\vec{wy}\in T$.
By definition, $W_{xy} = C(T)_{xy}$.
Moreover, Property (ii) of $y\in Y$ yields that $\vert W_{xy} \vert \ge d^+_{B_x^{\mp}}(y) - o(n) = \Omega(n)$.

By the $o(1)$-quasirandomness of $T$, there are $\Omega(n^2)$ pairs $(w,z) \in W_{xy} \times Z_{xy}$ that are oriented as $\vec{zw}$ in $T$. For such pairs, Claim~\ref{claim: bad cross edges imply many missing edges} yields that  $\vert\{wxy,wyz\} \cap M\vert\ge1$. We use this fact to prove that the codegree of $xy$ in $M$ must be large.
\begin{claim}\label{claim: missing degree of (x,y) huge}
$\vert M_{xy}\vert \geq \vert W_{xy}\vert - o(n)$  for every $y\in Y$.
\end{claim}

\begin{proof}
Given a vertex $y \in Y$, let $W'_{xy} \subseteq W_{xy}$ consist of those vertices $w\in W_{xy}$ satisfying:
\begin{enumerate}[(a)]
	\item $\vert M_{wy}\vert =o(n)$ and
	\item $d_T^-(w, Z_{xy})\ge {\vert Z_{xy}\vert}/{2} - o(n)$.
\end{enumerate}
By Property (iii) of $Y$, only $o(n)$ vertices in $W_{xy}$ can have linear degree in $M_y$. Further, the $o(1)$-quasirandomness of $T$ yields that only $o(n)$ vertices in $W_{xy}$ have in-degree from $Z_{xy}$ below $\vert Z_{xy}\vert/2 - o(n)$. Hence, $|W'_{xy}| = |W_{xy}| - o(n)$.

However, for each $w\in W'_{xy}$, there are $\Theta(\vert Z_{xy}\vert) = \Omega(n)$ vertices $z\in Z_{xy}$ such that we have both $\vec{zw}\in T$ (by (b)) and $wyz\notin M$ (by (a)). Therefore, $wxy\in M$ by Claim~\ref{claim: bad cross edges imply many missing edges} (or more precisely the consequence of Claim~\ref{claim: bad cross edges imply many missing edges} noted above Claim~\ref{claim: missing degree of (x,y) huge}), and thus $\vert M_{xy}\vert \geq \vert W'_{xy}\vert = \vert W_{xy}\vert - o(n)$ as claimed.
\end{proof}
\noindent As the next step in the proof, we show that most of the codegree of such pairs $xy$, where $y\in Y$, lies inside $B$.
\begin{claim}\label{claim: bad degree of (x,y) from Y are huge}
  $\vert B_{xy} \cap N_T^+(x) \vert = d(x,y) - o(n)$ for every $y\in Y$.
\end{claim}

\begin{proof}
Let $y\in Y$. Clearly, $\vert C(T)_{xy}\vert - \vert M_{xy}\vert+\vert B_{xy}\vert = d(x,y)$. Since  $C(T)_{xy}= W_{xy}$, Claim~\ref{claim: missing degree of (x,y) huge} readily yields that $\vert B_{xy}\vert = d(x,y) - o(n)$. Moreover, since $y$ is incident with $o(n)$ arcs in $B^-_x$ and $B^+_x$ (by Property (i) in the definition of $Y$), the vast majority of the codegree of the pair $xy$ must come from arcs incident to $y$ that lie in $B^\mp_x$. Since $y \in N_T^-(x)$, the other endpoints of those arcs are from $N_T^+(x)$.
\end{proof}
\noindent Since most of the codegree of $xy$ is in $B$, $G \cap C(T)$ cannot have many edges containing $xy$.

\begin{corollary}\label{corollary: intersection of xy-degree small}
$\vert(G \cap C(T))_{xy}\vert = o(n)$ for every $y\in Y$.\qed
\end{corollary}
\noindent We are now ready to prove that $|B_x^\mp| = o(n^2)$. By symmetry, it is enough to show that only $o(n)$ vertices from $N_T^-(x)$ can be incident to linearly many edges of $B^\mp_x$.
\begin{claim}\label{claim: few bad cross edges}
$|X_2^-|=o(n)$.
\end{claim}
\begin{proof}
Suppose for a contradiction that $|X_2^-| = \Omega(n)$. Thus, by Claim~\ref{claim: Y and X2- have same size}, we have $|Y|=\Omega(n)$. By Claim~\ref{claim: bad degree of (x,y) from Y are huge} and Property (i) of $Y$, at least $d(x,y)-o(n)\geq n/4-o(n)$ of the elements $z\in B_{xy}$ correspond to arcs $\vec{yz}\in B^{\mp}_x$ with $z\in N_T^+(x)$. We conclude that there are at least $|Y| \cdot (n/4 - o(n)) = \Omega(n^2)$ arcs from $Y \subseteq N_T^-(x)$ to $N_T^+(x)$.

Let $X_2^+ \subseteq N^+_T(x)$ be the set of vertices with in-degree $\Omega(n)$ in $B_x^\mp$. Clearly, $|X_2^+|=\Omega(n)$. Arguments analogous to those in the proofs of Claims~\ref{claim: Y and X2- have same size}-\ref{claim: bad degree of (x,y) from Y are huge} yield the existence of a subset $Z\subseteq X_2^+$ of size $|X_2^+| - o(n)$ such that every $z\in Z$ satisfies:
\begin{itemize}
\item $\vert d^+(z, N_T^-(x))- d^-(z, N_T^-(x))\vert = o(n)$,
\item $\vert (G \cap C(T))_{xz}\vert = o(n)$, and
\item $\vert M_{xz}\vert \ge  \vert U_{xz}\vert - o(n)$, where $U_{xz}$ is the collection of vertices $u\in N_T^-(x)$ with $\vec{zu}\in T$.
\end{itemize}
Now observe that the Property (ii) of $y\in Y$ can be rewritten as $\vert W_{xy} \vert = \vert N_T^+(x) \vert / 2 - o(n)$. Thus for all $y\in Y$ we have
\begin{align*}\vert N_T^+(x)\vert- \vert B_{xy}\cap N_T^+(x)\vert = \vert N_T^+(x)\setminus B_{xy}\vert  \geq \vert C(T)_{xy}\vert=\vert W_{xy}\vert \geq \frac{\vert N_T^+(x)\vert}2 - o(n)\,.
\end{align*} 
Rearranging the terms and applying Claim~\ref{claim: bad degree of (x,y) from Y are huge} and the codegree assumption yield
\begin{align*}\vert N_T^+(x)\vert\geq 2\vert B_{xy}\cap N_T^+(x)\vert +o(n)\geq 2d(x,y) +o(n)\geq n/2 - o(n).
\end{align*} 
A symmetric argument for a vertex $z \in Z$ yields that $\vert N_T^-(x)\vert \geq n/2 - o(n)$, whence both $N_T^+(x)$ and $N_T^-(x)$ have size $n/2+o(n)$.

By Corollary~\ref{corollary: intersection of xy-degree small}, only $o(n^2)$ pairs $(y,z) \in Y \times N^+_T(x)$ extend $x$ to an edge $xyz \in G \cap C(T)$.
On the other hand, by Property (i) of $Y$, for $y\in Y$ most of the codegree of the pair $xy$ in $B$, which is at least $n/4 - o(n)$ by Claim~\ref{claim: bad degree of (x,y) from Y are huge}, must be from arcs $\vec{yz} \in T$ with $z \in N_T^+(x)$. Since there are $n/4 + o(n)$ such arcs $\vec{yz}$ by Property (ii) of $Y$ and since, as proved above, $\vert N_T^+(x)\vert =n/2+o(n)$, we conclude that all but $o(n)$ vertices from $N_T^+(x)\cap N_T^+(y)$ lie in $B_{xy}$.

Consider now the tournament $T'$ obtained from $T$ by reversing the orientations of all the arcs $\vec{yx}$, where $y\in Y$, and all the arcs $\vec{xz}$, where $z \in Z$. This has the following effect on `good' and `bad' triples:
\begin{enumerate}[(1)]
	\item all the `bad pairs' $yz\in B_x$ with $(y,z) \in Y \times Z$ become `good' pairs, i.e., $xyz \in C(T')\cap G$,
	\item all the `good pairs' $yz\in (C(T)\cap G)_x$ with $y\in Y$ and $z\in N_T^+(x)$ or $y\in N_T^-(x)$ and $z\in Z$ become `bad' pairs with respect to $C(T')$, and
	\item all other `good' triples (i.e.\  all apart those corresponding to pairs from (2) above) from $C(T)\cap G$ remain in $C(T')\cap G$ .
\end{enumerate}
Note that there might be some new missing triples $xyz \in C(T') \setminus (G \cup C(T))$ but they are irrelevant for our argument. Our aim is to show that $\vert G \cap C(T') \vert\ge \vert G \cap C(T) \vert +\Omega(n^2)$ (thereby contradicting our assumption that $T$ maximises $\vert G \cap C(T) \vert $).

Indeed, by $o(1)$-quasirandomness of $T$, there are at least $\vert Y\vert \cdot \vert Z\vert/2 -o(n^2)=\Omega(n^2)$ arcs $\vec{yz}\in T$ with $(y,z) \in Y \times Z$. Moreover, all but $o(n^2)$ of those pairs $yz$ lie in $B_x$.

On the other hand, for all $y\in Y$ and $z\in Z$, Corollary~\ref{corollary: intersection of xy-degree small} and an analogous argument for $Z$ show  $\vert(G \cap C(T))_{xy}\vert =o(n)$ and $\vert(G \cap C(T))_{xz}\vert = o(n)$, and hence
\begin{align*}
	\vert C(T')\cap G\vert -\vert C(T)\cap G\vert &= \Omega(n^2) - \vert Y \vert \cdot \max_{y\in Y}\vert(G \cap C(T))_{xy}\vert - \vert Z \vert \cdot \max_{z\in Z}\vert(G \cap C(T))_{xz}\vert = \Omega(n^2)\ .\end{align*}
This contradicts our assumption that $T$ was a best fit tournament. Thus $\vert X_2^- \vert = o(n)$, as claimed.
\end{proof}
A symmetric argument yields that only $o(n)$ vertices from $N_T^+(x)$ can have a linear degree $B_x^\mp$. Therefore, $\vert B_x^\mp\vert = o(n^2)$. Since our choice of $x$ was arbitrary, this latter result together with Corollary~\ref{corollary: bx- and bx+ are small} allows us to conclude that all the link graphs of $B$ must be sparse.
\begin{corollary}\label{corollary:DeltaG-CT}
For every $x\in [n]$, we have $|B_x| = o(n^2)$. \qed
\end{corollary}
\noindent We now show that $T$ must be asymptotically balanced and that $M$ has low vertex-degrees, i.e., that for every vertex $x$ the in-degree and the out-degree of $x$ in $T$ are almost equal, and that there are few triples of $M$ containing $x$. 
\begin{claim}\label{claim: T is almost balanced}
For every $x\in [n]$, $|d_T^+(x) - d_T^-(x)| = o(n)$ and $|M_x| = o(n^2)$.
\end{claim}

\begin{proof}
The link graph $G_x$ has minimum degree $n/4 - o(n)$, which it inherits from $\delta_2(G) = n/4 - o(n)$. In particular, $|G_x| \ge n^2/8 - o(n^2)$. On the other hand, $\vert B_x\vert = o(n^2)$ by Corollary~\ref{corollary:DeltaG-CT}.  Finally, the $o(1)$-quasirandomness of $T$ yields that
\[\vert C(T)_x\vert = \frac{1}{2} \cdot d_T^+(x)(n-1- d_T^+(x))+o(n^2) \le \frac{n^2}8 + o(n^2).\]
Since \[\vert G_x\vert = \vert C(T)_x\vert -\vert M_x\vert +\vert B_x \vert , \]
both parts of the statement follows.
\end{proof}
\noindent We now show that for every pair of distinct vertices $x,y$ the codegree of $xy$ in $B$ is small.
\begin{claim}\label{claim:CodegreeG-CT} 
For every pair of distinct vertices $x,y$ in $[n]$, we have $\vert B_{xy}\vert = o(n)$.
\end{claim}
\begin{proof}
Fix $x\in [n]$. Suppose there exists a vertex $y$ such that $\vert B_{xy}\vert = \Omega(n)$, and, without loss of generality, suppose that $\vec{xy}$ in $T$. Let us divide the vertices in $B_{xy}$ into three sets: 
\begin{eqnarray*}
	B_{xy}^+&:=& \{z\in B_{xy}:\ \{\vec{xz},\vec{yz}\} \subseteq T\},\\
	B_{xy}^-&:=&\{z\in B_{xy}:\ \{\vec{zx},\vec{zy}\} \subseteq T\},\hbox{ and}\\
	B_{xy}^{\pm} &:= & B_{xy} \setminus \left( B_{xy}^- \cup B_{xy}^+\right)=\{z\in B_{xy}:\ \{\vec{xz},\vec{zy}\} \subseteq T\}.
\end{eqnarray*}
It is enough to show that all three sets have size $o(n)$.

Suppose first that $\vert B_{xy}^+\vert = \Omega(n)$. Let  $W_{xy}^-$ be the collection of $w\in [n]$ such that $\{\vec{wx}, \vec{wy} \}  \subseteq T$. By Claim~\ref{claim: T is almost balanced} and the inclusion-exclusion principle, we have
\begin{align*}
\vert W_{xy}^-\vert \geq n-2- d^+(x)-d^+(y) + \vert B_{xy}^+\vert &\ge \Omega(n).
\end{align*}
By $o(1)$-quasirandomness of $T$, there are at least $\frac{1}{2} \vert W_{xy}^-\vert \cdot \vert B_{xy}^+\vert -o (n^2)$ pairs $(w,z)\in W_{xy}^-\times B_{xy}^+$ such that $\vec{zw}\in T$. For each such a pair, both $wxz$ and $wyz$ lie in $C(T)$. However, since $xyz\in B$ and $G$ is $K_4^-$-free, at least one of the triples $wxz, wyz$ must lie in $M$, and hence 
\begin{align*}
{\vert M_x\vert + \vert M_y\vert} & \ge \frac{\vert W_{xy}^-\vert \cdot \vert B_{xy}^+\vert}{2} - o(n^2) = \Omega(n^2),
\end{align*}
contradicting the vertex-degree bound in $M$ established in Claim~\ref{claim: T is almost balanced}. Therefore, $\vert B_{xy}^+ \vert = o(n)$.  A~symmetric argument yields that  $\vert B_{xy}^- \vert = o(n)$.

Finally, we turn our attention to the collection $B_{xy}^{\pm}$.  Suppose, for the sake of a contradiction, that $\vert B_{xy}^{\pm}\vert = \Omega(n)$. Then, by (\ref{eq:Reverse}), we have also $\vert (G\cap C(T))_{xy}\vert = \Omega(n)$. Fix a pair of vertices $(w,z) \in (G\cap C(T))_{xy} \times B_{xy}^{\pm}$.

If $\vec{zw}\in T$, then $wxz$ induces a cyclically oriented triangle in $T$ which must be a missing triple; indeed, note that $xyw, xyz\in G$. Similarly if $\vec{wz}\in T$, then $ywz \in M$. Analogously to the case of bounding $|B^+_{xy}|$, $x$ or $y$ must have its degree in $M$ being at least
\[\frac{\vert B_{xy}^{\pm}\vert \cdot \vert (G\cap C(T))_{xy}\vert}2 = \Omega(n^2)\]
contradicting Claim~\ref{claim: T is almost balanced}.
\end{proof}
\noindent With these results in hand, we are now ready to prove the crucial claim needed to finish the proof of Theorem~\ref{theorem: largest size among codegree extremal achieved by tournaments}.
\begin{claim}\label{claim:CT-GLocal}
For every bad edge $E_B\in B$ there are at least $n/4 - o(n)$ missing edges $E_M\in M$ such that $\vert E_B\cap E_M\vert = 2$.
\end{claim}

\begin{proof}
Let $E_B=xyz\in B$ be an arbitrary bad edge. Without loss of generality, the pairs from $E_B$ are oriented as $\vec{xy}, \vec{yz}, \vec{xz}$ in $T$. Now consider a vertex $w\in C(T)_{xz}$.

By definition, we have $\vec{zw}, \vec{wx}\in T$. If $\vec{wy}\in T$ then $\{wxz,wyz\} \subseteq C(T)$, and since $G$ is $K_4^-$-free and $xyz\in G$, at least one of the triples $wxy, wxz$ must lie in $M$. Similarly, if $\vec{yw}\in T$ then $\{wxz,wxy\} \subseteq C(T)$ and at least one of the triples lies in $M$. In either case, each $w\in C(T)_{xy}$ yields a triple containing $w$ that lies in $M$ and intersects $E_B$ in two vertices. However, by the codegree assumption on $G$ and Claim~\ref{claim:CodegreeG-CT}, we have
\begin{align*}
\vert C(T)_{xz}\vert &\geq \vert G_{xy}\vert -\vert B_{xy}\vert\ = n/4 - o(n)
\end{align*}
choices of such a vertex $w$. This finishes the proof of the claim.
\end{proof}
Let us now finish the proof of Theorem~\ref{theorem: largest size among codegree extremal achieved by tournaments}.
Provided $n$ is sufficiently large, Claim~\ref{claim:CT-GLocal} yields at least $\vert B\vert \cdot \frac n5$ pairs $(E_B, E_M)\in B\times M$ with $|E_B \cap E_M| = 2$. Moreover, each $E_M=xyz\in M$ can feature in at most $\vert B_{xy}\vert + \vert B_{yz}\vert +\vert B_{xz}\vert$ of these pairs, which by Claim~\ref{claim:CodegreeG-CT} is at most $o(n)$. Therefore,
\[o(n) \cdot \vert M\vert \ge \vert B\vert \cdot \frac{n}{5}\ ,\] 
thus if there is at least one bad edge, then $\vert M\vert >\vert B\vert$.
We conclude that $\vert C(T)\vert \geq \vert (G)\vert $ with equality only if $G=C(T)$.
\qed
\section{Concluding remarks}\label{section: conclusion}
\subsection{Better lower bounds for $\ex_2(n, K_4^-)$ ?}
We have shown in Theorem~\ref{theorem: codegree threshold} that $\mathrm{ex}_2(n, K_4^-)\leq \lfloor\frac{n+1}{4}\rfloor$, and that this upper bound is tight when $n = 4k+3$ or $4k+2$ and there exists a skew Hadamard matrix of order $4k+4$. Short of proving Seberry's (and hence Hadamard's) conjecture, what is the best general lower bound one can give on $\ex_2(n, K_4^-)$ ?

Taking a tournament $T$ on $[n]$ uniformly at random and considering the Erd{\H o}s--Hajnal construction $C(T)$ yields a lower-bound of $\frac{n}{4}-O\left(\sqrt{n\log n}\right)$ via a standard Chernoff bound on the probability that a given pair has low codegree and a union bound over all pairs. One can do a little better, however, by exploiting results about the distribution of primes congruent to $3 \pmod 4$.

Indeed, suppose that $c>0$ is such that for every $n$ sufficiently large, there exists an integer $m$ with $n< m \leq n+n^c$ such that there exists a skew Hadamard matrix of order $m$.  Then, as shown in Proposition~\ref{prop:  existence of good tournaments},  there is a tournament $T$ on $m-1$ vertices such that $\delta_2(C(T))= \frac{m}{4}$. Now delete a set of $(m-1-n)$ vertices of $T$ chosen uniformly at random to obtain a new tournament $T'$ on $n$ vertices exactly. Taking Chernoff bounds for the hypergeometric distribution (see e.g.\ ~\cite[Lemma 2]{HaggkvistThomason95}) and a union bound over all pairs yields that w.h.p.\ 
\begin{align*}
\ex_2(n,K_4^-)\geq \delta_2(C(T'))\geq \frac{n}{4}-O\left(\sqrt{(m-n)\log n }\right)=\frac{n}{4}-O\left(n^{c/2+o(1)}\right).
\end{align*}
As stated in the introduction, Seberry's conjecture is known to hold for all values of $m$ of the form $m=2^t \prod_{i \in I} (q_i+1)$, where $t\in \mathbb{Z}_{\geq 0}$ and $q_i$ is a prime power congruent to $3\pmod 4$ for all $i\in I$ (see \cite[Theorem 4.1]{WallisStreetSeberry72}). In particular it holds for all $m$ such that $m-1$ is a prime congruent to $3 \pmod 4$. Now it is known~\cite[Theorem~3(I)]{BakerHarmanPintz96} that for $n\in \mathbb{N}$, there exists a prime $m-1$ congruent to $3 \pmod 4$ in the interval $n\leq  m-1 \leq n +n^{0.55+o(1)}$. Together with the argument in the paragraph above, this implies
\begin{align*}
\ex_2(n, K_4^-)\geq \frac{n}{4} -O\left(n^{0.275+o(1)}\right).
\end{align*}
We believe however that the correct bound should be of the form $n/4-O(1)$, and we propose the following problem, which can be viewed as a weakening of Seberry's conjecture.
\begin{problem}
Show that there exists a constant $C>0$ such that for any $n\in \mathbb{N}$ there exists a tournament $T$ on $n$ vertices with $\delta_2(T)\geq \frac{n}{4}-C$.
\end{problem}

\subsection{Codegree density and smooth Tur\'an density}

As mentioned in Section~\ref{subsection: background}, the codegree density of $K_4^-$ coincides with various `smooth' or `weakly quasirandom' versions of Tur\'an density, and the near-extremal constructions are the same. What is more, the conjectured value of the still unknown codegree density of $K_4$ is the same as the values of several of its `weakly quasirandom' Tur\'an densities, as shown by Reiher, R\"odl and Schacht~\cite{ReiherRodlSchacht16a, ReiherRodlSchacht16} (see also~\cite{ReiherRodlSchacht16d, ReiherRodlSchacht16c}). Again the (random) extremal constructions are the same. It is natural to ask whether there is any relationship in general between the codegree density of a $3$-graph and its `smooth' Tur\'an densities. The $3$-graph $F_{3,2}=\{abc, abd, abe,cde\}$ has $\pi_2(F_{3,2})=1/3$, as proved in~\cite{FalgasRavryMarchantPikhurkoVaughan15}, but has trivial, zero density for all notions of smooth Tur\'an density, so such a relationship would have to take the form of a one-sided upper bound.

To make this discussion more precise, let us consider the weakest notion of weak quasirandomness: a $3$-graph $G$ on $n$-vertices is said to be $K_{1,1,1}$-\emph{weakly quasirandom} (wqr) with parameters $(\varepsilon, p)$ if for all subsets $X\subseteq V(G)$, $\left\vert \vert G[X] \vert - p \binom{\vert X\vert} {3}\right\vert \leq \varepsilon n^3$.  A sequence of $3$-graphs $(G_n)_{n\in \mathbb{N}}$ with $\vert V(G_n)\vert \rightarrow \infty$ as $n\rightarrow \infty$ is then said to be $K_{1,1,1}$\emph{-wqr with density} $p$ if each $G_n$ is $K_{1,1,1}$-wqr with  parameters $(p, \varepsilon_n)$ and $\varepsilon_n\rightarrow 0$. The $K_{1,1,1}$\emph{-wqr Tur\'an density} of a $3$-graph $F$ is the infimum over all $p\geq 0$ such that in every $K_{1,1,1}$-wqr sequence $(G_n)_{n\in \mathbb{N}}$ with density at least $p$, all but finitely many of the $G_n$ contain a copy of $F$ as a subgraph. Stronger notions of weak quasirandomness and associated Tur\'an densities (which are lower bounded by the $K_{1,1,1}$-wqr Tur\'an density) exist, see~\cite{ReiherRodlSchacht16}. Our question above in its strongest form thus asks:

\begin{question}\label{question: codegree density upper bound for smooth turan density}

Is the codegree density of a $3$-graph $F$ always an upper bound on its $K_{1,1,1}$-wqr Tur\'an density?

\end{question}

Shortly after the first draft of this paper was written, Falgas-Ravry and Lo~\cite{FalgasRavryLo16} gave a positive answer to Question~\ref{question: codegree density upper bound for smooth turan density} when $K_{1,1,1}$-wqr is replaced by a slightly stronger form of weak quasirandomness. One reason to believe the answer to Question~\ref{question: codegree density upper bound for smooth turan density} might be positive is that it would be enough to show that one can extract from any $K_{1,1,1}$-wqr $3$-graph with density $p$ a `large' subgraph in which `most' of the pairs have codegree density $p$. This follows from an argument of~\cite{FalgasRavryLo16}; we give here a simpler proof of this fact tailored to the present setting.

\begin{proposition}\label{proposition: codegree density with a few bad pairs allowed}

Let $F$ be a $3$-graph and $\varepsilon>0$ be fixed. Then there exist $\eta>0$ and $n_0\in \mathbb{N}$ such that if  $G$ is a $3$-graph on $n\geq n_0$ vertices in which all but $\eta n^2 $ pairs of vertices have codegree at least $(\pi_2(F)+\varepsilon)n$ then $G$ must contain a copy of $F$ .

\end{proposition}

\begin{proof}

Let $G$ be a $3$-graph in which all but $\eta n^2 $ pairs of vertices have codegree at least $(\pi_2(F)+\varepsilon)n$. Add at most $\eta n^3$ triples to $G$ to obtain a new $3$-graph $G'$ on $n$ vertices with $\delta_2(G')\geq 	(\pi_2(F)+\varepsilon)n$. Provided $n_0$ is sufficiently large, $\delta_2(G')>\ex_2(n, F) +\frac{\varepsilon}{2} n$.

As shown by Mubayi and Zhao in~\cite[Proposition~1.4]{MubayiZhao07}, the codegree density function exhibits \emph{supersaturation}, which in our case implies that $G'$ contains $C_{\varepsilon}n^{v(F)}$ copies of $F$, where $C_{\varepsilon}>0$ is a constant depending on $\varepsilon$. Now let us remove the triples from $G'\setminus G$. Each of these is adjacent to at most $v(F)! \cdot n^{v(F) -3}$ copies of $F$ in $G'$. Thus there must remain at least $(C_{\varepsilon}- \eta v(F) !)\cdot n^{v(F)}$ copies of $F$ in $G$. Taking $\eta=\eta(\varepsilon, F)$ sufficiently small, this is strictly positive, establishing our claim.
\end{proof}

\noindent Proposition~\ref{proposition: codegree density with a few bad pairs allowed} says, in essence, that allowing for $o(n^2)$ pairs to have low codegree does not change the codegree threshold by more than $o(n)$. Finally, let us note that embarrassingly we do not know the answer to the following question:

\begin{question}\label{question: high codegree subgraphs in quasirandom subsequence}

Let $(G_n)_{n\in \mathbb{N}}$ be a sequence of $K_{1,1,1}$-wqr $3$-graphs with density $p>0$.
Must there exist a sequence of subgraphs $(H_n)_{n\in \mathbb{N}}$ with $H_n\subseteq G_n$, $v(H_n)\rightarrow \infty$ and $\delta_2(H_n)/v(H_n)$ bounded away from zero?

\end{question}

In other words, does a smooth distribution of the edges in a $3$-graph imply the existence of a reasonably large subgraph with high internal codegree density? Although the instinctive reaction of many researchers we have consulted was that the answer to Question~\ref{question: high codegree subgraphs in quasirandom subsequence} should be negative, a counterexample seems hard to come by: the quasirandomness condition pulls in the direction of random constructions, where it is difficult to control the codegree of subgraphs on $O(\log n)$ vertices, while the necessity to have no large subgraph with large internal codegree density pushes us towards structured constructions, which fail to be sufficiently quasirandom. Whether positive or negative, we expect an answer to Question~\ref{question: high codegree subgraphs in quasirandom subsequence} would thus be quite interesting.

We note that Falgas-Ravry and Lo~\cite{FalgasRavryLo16} gave a positive answer to Question~\ref{question: high codegree subgraphs in quasirandom subsequence} albeit again only for a stronger notion of weak quasirandomness than $K_{1,1,1}$-wqr. However, they also provided examples of $K_{1,1,1}$-wqr graphs in which the order of a largest subgraph with non-zero minimum codegree only grew as the inverse of the error parameter $\varepsilon$, suggesting that the situation is not entirely obvious one way or the other. 

\subsection{Tur\'an problems with codegree conditions}

The codegree-extremal construction for $K_4^-$ is quite different from the recursive construction of Frankl and F\"uredi which is conjectured to be extremal for the Tur\'an problem (indeed the later construction has pairs of vertices with codegree zero). A natural question is what happens if we interpolate between the two problems. Explicitly, for $c\in[0, \pi_2(F)]$, let $\ex^{\delta_2\geq c}(n, F)$ denote the maximum number of edges in an $F$-free $3$-graph on $n$ vertices with minimum codegree at least $c(n-2)$ (if such a $3$-graph exists). 

\begin{problem}\label{problem: Turan problem with codegree condition}

Determine the asymptotic behaviour of $\ex^{\delta_2\geq c}(n, K_4^-)$ for $c\in [0,1/4]$.

\end{problem}
\noindent Write $f(c)$ for the limit of  $\ex^{\delta_2\geq c}(n, K_4^-)/\binom{n}{3}$ as $n\rightarrow \infty$ (it can be shown this limit exists). The function $f(c)$ is nonincreasing in $[0, \pi_2(F)]$ and takes values in $[0,\pi(F)]$. We have shown $f(1/4)=1/4$, and Mubayi conjectured in~\cite[Conjecture 2]{Mubayi03} $f(0)=2/7$. We can give a lower bound for $f(c)$ in the range $(0, 1/4)$ as follows.

Let $H_6$ denote the $3$-graph on $[6]$ with edges $\{123, 234, 345, 145, 125, 136, 356, 256, 246, 146\}$. This $3$-graph was constructed by Frankl and F\"uredi~\cite{FranklFuredi84} and has the property that the link graph of every vertex is a cycle on $5$-vertices; in particular, $H_6$, its blow-ups and its iterated blow-ups are all $K_4^-$-free. Given $c>0$, let $t =\lfloor \frac{\log (1/4c)}{\log 6}\rfloor$. Add all triples from a balanced blow-up of $H_6$ on $n$ vertices, then repeat this construction inside each of the $6$ parts, iterating this procedure a total of $t$ times. Finally, inside each of the $6^t$ parts of size $(1+o(1))n/6^t$ that remain, place a $C(T)$ construction. This gives a $K_4^-$-free $3$-graph with minimum codegree $(1/4+o(1))n/6^t\geq (c+o(1))n$ and $(2/7 - 6^{-2t}/28 +o(1))\binom{n}{3}$ edges.

\begin{corollary}\label{corollary: lower bound for Turan problem with codegree condition}

$f(c)\geq \frac{2}{7} -\frac{1}{28} 6^{-2\lfloor \log (1/4c)/\log 6\rfloor}$. \qed

\end{corollary}
\noindent Note that our lower bound is not a continuous function of $c$; however, as it does not seem possible to shift in a continuous way from a $C(T)$ (codegree-extremal) construction to an iterated blow-up of $H_6$ (the conjectured extremal construction for Tur\'an density), this is a behaviour we could plausibly encounter.

\pagebreak

\appendix\section{Description of the supplementary computer programs}

We split the task of a formal verification of our proof of Lemma~\ref{lemma: flag} into six steps (A.1-A.6),
where each step will be accompanied with a small script written in SAGE~\cite{SAGEmath} that performs the described task.
Note that we also have independent implementations in C of each of the tasks listed below, which do run faster, 
however, they are lengthier and verifying their correctness is somewhat more tedious.

Before we continue, let us briefly recall some notation from Section~\ref{section: flag algebras}. For two $3$-graphs $H$ and $G$, $p(H,G)$ denotes the probability that a random $v(H)$-vertex subgraph of $G$ is isomorphic to $H$.
Also recall that for a $k$-vertex type $\sigma$ and a fixed $\sigma$-flag $F$, the value $p_F^\sigma$ denotes the probability that a random injection from $[k]$ to $F^\emptyset$ yields a $\sigma$-flag that is isomorphic to $F$.
Finally, for two $\sigma$-flags $F_1$ and $F_2$, and a $(v(F_1)+v(F_2)-v(\sigma))$-vertex $\sigma$-flag $H$, $p(F_1,F_2,H)$ denotes
the probability that a random partition of the unlabelled vertices of $H$ into two parts of respective sizes $v(F_1)-v(\sigma)$ and $v(F_2)-v(\sigma)$ yields $\sigma$-flags isomorphic to $F_1$ and $F_2$.

To simplify the following presentation, let us define the so-called \emph{flag pair density}.
Given a pair of $\sigma$-flags $F_1$ and $F_2$, and a $3$-graph $G$ with $v(F_1)+v(F_2)-v(\sigma)$ vertices, the flag pair density $\wbar{p}(F_1,F_2,G)$ is defined as 
\[
\sum_{\substack{F \in \cF_{v(G)}^{\sigma} \\ F^\emptyset = G}} p^{\sigma}_F \cdot p(F_1,F_2,F)
.\]
For a $3$-graph $G$ satisfying $v(G) > v(F_1)+v(F_2)-v(\sigma)$, we generalize the above and define the flag pair density $\wbar{p}(F_1,F_2,G)$ as follows:
\[
\wbar{p}(F_1,F_2,G) := \sum_{H \in \cF_{v(F_1)+v(F_2)-v(\sigma)}} p(H,G) \cdot \wbar{p}(F_1,F_2,H)
.\]

\subsection{The list of $7$-vertex $K_4^-$-free $3$-graphs}
We identify $3$-graphs on the vertex-set $[k]$ with the collection of their edges.
We order the collection of such $3$-graphs in the lexicographic order inherited from the order on $[k]$.

Given $k$, we generate the list of $k$-vertex $K_4^-$-free $3$-graphs on the vertex-set $[k]$ that are lexicographically minimal within their isomorphism class in the following way:
we iteratively go over all the labelled $k$-vertex $K_4^-$-free $3$-graphs $G$ such that the vertex-subset $[k-1]$ induces some lexicographically minimal $(k-1)$-vertex $K_4^-$-free $3$-graph,
and include in the list exactly those $3$-graphs $G$ that are lexicographically minimal.
We note that our SAGE script outputs the list in the same order and format as Flagmatic 1.5.1 does.

For $k=7$, generating this list readily yields that $|\cF_7| = 8157$.
Moreover, the maximum number of edges of a $7$-vertex $K_4^-$-free $3$-graph is $15$,
and the number of $3$-graphs in~$\cF_7$ with exactly $2$, $3$, $4$,\dots, $14$ and $15$ edges
is equal to $3$, $9$, $32$, $102$, $304$, $752$, $1451$, $2022$, $1909$, $1118$, $374$, $70$, $8$ and $1$, respectively.

\smallskip

\noindent SAGE script: {\tt gen7.sage}

\noindent Output file: {\tt graph\_lists/list7}

\subsection{The lists of $6$-vertex $\iota_i$-flags, $i \in [6]$, and the corresponding pair densities}
Firstly, for every $F \in \cF_6$ we consider all possible injections of $[5]$ into $V(F)$ and generate the appropriate lists $\cF_6^{\iota_i}$.
Then, for every $G \in \cF_7$, we consider all the injections $m: [5] \to V(G)$ together with the partitions of $V(G) \setminus m([5])$ into two non-empty parts.
This allows us to directly compute the coefficients $\wbar{p}(F_1,F_2,G)$ for every expression of the form
\[\unlab{F_1 \times F_2}{\iota_i}
= \sum_{F \in \cF_7^{\iota_i}} p^{\iota_i}_F \cdot p(F_1,F_2,F) \cdot F^\emptyset 
= \sum_{G \in \cF_7} \wbar{p}(F_1,F_2,G) \cdot G.\]
As in the previous subsection, we output this data in the same format as Flagmatic 1.5.1 does.
\smallskip

\noindent SAGE script: {\tt gen\_flags.sage}

\noindent Output files: {\tt flagmatic\_flags-pruned.rat} and {\tt graph\_lists/list6\_iotaX} for $X \in [6]$

\subsection{The codegree expressions $D \in \cD$ as linear combinations of $\cF_7$}
Similarly to the previous section, we first generate the list $\cF_6^\tau$ by considering all the possible injections of $[2]$ into the elements of $\cF_6$.
For each $F \in \cF_6^\tau$ we then compute the coefficients $\wbar{p}(F,E^\tau,G)$ and $\wbar{p}(F,N^\tau,G)$ for the expressions
\[\unlab{F \times E^\tau}{\tau} = \sum_{G \in \cF_7} \wbar{p}(F,E^\tau,G) \cdot G \quad \mbox{and}\quad \unlab{F \times N^\tau}{\tau} = \sum_{G \in \cF_7} \wbar{p}(F,N^\tau,G) \cdot G.\]
As we have noted in Section~\ref{subsection: codegree density}, although the set $\cF_6^\tau$ has size $1643$, the symmetry of the type $\tau$ yields that $|\cD| = 905$.

\smallskip

\noindent SAGE script: {\tt codeg.sage}

\noindent Output file: {\tt ineq\_codeg}

\subsection{The tight-path expressions $\cP_0$, $\cP_1$ and $\cP_2$  as linear combinations of $\cF_7$}
For the $4$-vertex types $\sigma_i$ with $i\in\{0,1,2\}$ edges, we generate the lists $\cF_5^{\sigma_i}$ by considering all the possible injections of $[4]$ into the elements of $\cF_5$,
and then identify the corresponding elements of $\cP_i \subseteq \cF_5^{\sigma_i}$.
Next, in a similar fashion as in the two previous sections, we compute for every $F_1,F_2\in\cF_5^{\sigma_i}$ and $G \in \cF_7$ the coefficients $\wbar{p}(F_1,F_2,G)$ in
\[\unlab{F_1 \times F_2}{\sigma_i}
= \sum_{F \in \cF_6} \wbar{p}(F_1,F_2,F) \cdot F
= \sum_{F \in \cF_6} \wbar{p}(F_1,F_2,F) \cdot \sum_{G \in \cF_7} p(F,G) \cdot G
= \sum_{G \in \cF_7} \wbar{p}(F_1,F_2,G) \cdot G .\]

\smallskip

\noindent SAGE script: {\tt tightpath.sage}

\noindent Output file: {\tt ineq\_tightpath}

\subsection{The size of the set $\cE_T \cap \cF_7$ is equal to $247$}
First, we generate all the $456$ non-isomorphic tournaments on $7$-vertices.
Next, for each such a tournament $T$, we construct the $3$-graph $C(T)$, and check whether it is isomorphic to $C(T')$ for some other tournament $T'$.
It turns out that $|\cE_T \cap \cF_7| = 247$.
Observe that for the sole purpose of verifying~\eqref{equality: flag identity}, it is enough to compute the size of $\cE_T \cap \cF_7$.
Indeed, any identity of the form \eqref{equality: flag identity} may have a positive slack on at most $|\cF_7 \setminus \cE_T| = 7910$ coordinates by complementary slackness,
and the identity for our particular choice of $Q_i$, $I_i$, $c_j$ and $u_D$, where $i \in [6]$, $j\in [3]$ and $D \in \cD$, has a positive slack on exactly $7910$ coordinates.

\smallskip

\noindent SAGE script: {\tt ext7.sage}

\subsection{Verifying the identity~\eqref{equality: flag identity} stated in Lemma~\ref{lemma: flag}}
Given the input data for the entries of the matrices $I_i$ and $Q_i$, where $i \in [6]$,
the positive rationals $c_0$, $c_1$, $c_2$ and $u_D$ for $D \in \cD$, we first check the positive definiteness of every $Q_i$.
Next, we compute the right-hand side of~\eqref{equality: flag identity} and compare its $8157$ coefficients one by one with the left-hand side of~\eqref{equality: flag identity}. 
Finally, we verify that there are exactly $7910$ coordinates with a positive slack.

\smallskip

\noindent SAGE script: {\tt lemma28.sage}

\end{document}